\documentclass[a4paper,12pt]{amsart}

\newfont{\cyr}{wncyr10 scaled 1100}

\usepackage[left=3cm,right=3cm,top=3.5cm,bottom=3cm]{geometry}
\usepackage{amsthm,amssymb,amsmath,amsfonts,mathrsfs,amscd,yfonts}
\usepackage{turnstile}
\usepackage[latin1]{inputenc}
\usepackage[all]{xy}
\usepackage{latexsym}
\usepackage{longtable}
\usepackage[usenames,dvipsnames]{color}
\usepackage{hyperref}

\theoremstyle{plain}
\newtheorem{theorem}{Theorem}[section]
\newtheorem{corollary}[theorem]{Corollary}
\newtheorem{lemma}[theorem]{Lemma}
\newtheorem{proposition}[theorem]{Proposition}

\theoremstyle{definition}
\newtheorem{definition}[theorem]{Definition}

\newtheorem{examplewr}[theorem]{Example}

\theoremstyle{remark}
\newtheorem{obswr}[theorem]{Observation}
\newtheorem{remarkwr}[theorem]{Remark}

\newenvironment{remark}{\begin{remarkwr}\begin{upshape}}{\end{upshape}\end{remarkwr}}

\DeclareMathOperator{\Reg}{Reg}
\DeclareMathOperator{\Ind}{Ind}

\DeclareMathOperator{\ad}{ad}

\newcommand{\Q}{\mathbb{Q}}
\newcommand{\Z}{\mathbb{Z}}

\newcommand{\C}{\mathbb{C}}

\newcommand{\Gal}{\mathrm{Gal\,}}
\newcommand{\GL}{\mathrm{GL}}

\newcommand{\Frob}{\mathrm{Fr}}
\newcommand{\End}{\mathrm{End}}

\newcommand{\Fr}{\mathrm{Fr}}

\newcommand{\ord}{{\mathrm{ord}}}
\newfont{\gotip}{eufb10 at 12pt}

\newcommand{\cO}{{\mathcal O}}

\newcommand{\cL}{{\mathcal L}}

\newcommand{\lra}{\longrightarrow}

\newcommand{\Norm}{\mathrm{N}}

\DeclareMathOperator{\Hom}{Hom}

\newcommand{\fp}{{\mathfrak p}}

\usepackage{tikz-cd}

\begin{document}

\title{On the $\cL$-invariant of the adjoint of a weight one modular form}

\author{Marti Roset, Victor Rotger and Vinayak Vatsal}

\address{M.R.: Universit\'e Paris-Sud, B\^atiment 425, 91405 Orsay, France}
\email{marti.roset-julia@u-psud.fr}
\address{V. R.: Departament de Matem\`{a}tica Aplicada II, Universitat Polit\`{e}cnica de Catalunya, C. Jordi Girona 1-3, 08034 Barcelona, Spain}
\email{victor.rotger@upc.edu}
\address{V.V.: Deptartment of Mathematics, University of British Columbia, BC V6T 1Z2 Vancouver, Canada
}
\email{vatsal@math.ubc.ca}


\maketitle
\centerline\today

\begin{abstract}
	The purpose of this article is proving the equality of two natural $\mathcal L$-invariants attached to the adjoint representation of a weigth one cusp form, each defined by purely analytic, respectively algebraic means. The proof departs from Greenberg's \cite{Gr} definition of the algebraic $\mathcal L$-invariant as a universal norm of a canonical $\Z_p$-extension of $\Q_p$ associated to the representation. We relate it to a certain $2\times 2$ regulator of $p$-adic logarithms of global units by means of class field theory, which we then show to be equal to the analytic $\mathcal L$-invariant computed in \cite{RR1}. 
\end{abstract}

\tableofcontents

\section{Introduction}\label{SectionIntroduction}

Let $g = \sum_{n\geq 1} a_n q^n  \in S_1(N,\eta)$ be a normalized newform of weight one, level $N$ and nebentype character $\eta$, with Fourier coefficients in a finite extension $L$ of $\Q$. Let $p$ be a rational prime not dividing $N$. 
The purpose of this article is proving the equality of two natural $\cL$-invariants that may be associated to the  adjoint representation $\mathrm{ad}^0(g)$ of $g$, each defined by purely analytic, respectively algebraic means. 

That these two $\cL$-invariants should be equal was already conjectured by Greenberg \cite{Gr} and is indeed natural to expect in view of the main conjecture for $\mathrm{ad}^0(g)$, which still remains unsolved in this scenario. 

To put our question in context, recall that a problem of the same flavour already takes place for modular forms of higher weight and level divisible by $p$, where several, a priori inequivalent $\cL$-invariants had been introduced by Mazur-Tate-Teitelbaum, Breuil, Coleman, Fontaine and Mazur, Darmon and Orton: cf.\,\cite{Col} for an account of this fascinating subject; nowadays all these $\cL$-invariants in weight $k\geq 2$ are known to be equal thanks to the work of many authors.  Ever more germane to our setting, the analogous question for the adjoint of cusp forms of weight $k \geq 2$ was resolved by the work of Citro and Dasgupta (\cite{Ci}, \cite{Da})  by methods that are very different from ours. The main substantial difference  is that the $\cL$-invariants arising in the above scenarios take place at a central critical point and are of local nature as they only depend on the restriction to $G_{\Q_p}$ of the underlying Galois representation, while the $\cL$-invariant of this note is global and the associated $L$-function has no critical points. Our setting is rather closer in spirit to Benois' fundamental work on $\cL$-invariants at non-central critical points: cf.\,\cite{Benois} and Horte's recent Ph.D thesis \cite{Horte}, although we do not take this approach in our article.

In order to state explicitly our main theorem, fix embeddings $\bar{\Q} \subset \C$ and $\bar{\Q} \subset \bar{\Q}_p$, let $\tau \in G_\Q$ be complex conjugation and let $I_p \subset G_{\Q_p} \subset G_\Q$ be the corresponding inertia and decomposition subgroups, respectively. Denote by
\[
\varrho_g: \Gal(H_g/\Q) \hookrightarrow \GL(V_g) \simeq \GL_2(L)
\]
the odd Artin representation associated to $g$, where $H_g$ is a finite Galois extension of $\Q$ and $V_g$ is a $2$-dimensional vector space over $L$. 

Set $g^* := g \otimes \eta^{-1} = \sum_{n\geq 1} a^\tau_n q^n  \in S_1(N,\bar\eta)$ and note that $V_{g^*}$ is the contragredient dual of  $V_g$. Conjugation by $\varrho_g$ gives rise to a linear action of $G_\Q$ on $\End(V_g) \simeq V_g \otimes V_g^*$ and we denote by $\ad^0(g)$ the $G_\Q$-equivariant subspace of $\End(V_g)$ corresponding to null-trace endomorphisms. Let
\[
\varrho_{\mathrm{ad}^0(g)}: \Gal(H/\Q) \hookrightarrow \GL(\mathrm{ad}^0(g)) \simeq \GL_3(L)
\]
be the associated {\em adjoint} representation. Here $H$ is a Galois extension sitting in $H_g$. 

We begin by explaining the analytic side of the story. Label and order the roots of the $p$th Hecke polynomial of $g$ as $X^2-a_pX+\eta(p) = (X-\alpha)(X-\beta)$. Let
$$
g_\alpha(q) = g(q) -\beta g(q^p)
$$
denote the $p$-stabilization of $g$ on which the Hecke operator $U_p$ acts with eigenvalue $\alpha$. Note that the roots of the $p$th Hecke polynomial of $g^*$ are $1/\alpha$, $1/\beta$ so we can also consider the $p$-stabilization $g^*_{1/\beta}$. 

Let $L_p(\ad^0(g_\alpha),s)$ be Hida-Schmidt's $p$-adic $L$-function associated to the adjoint of $g_\alpha$. It has an exceptional zero at $s = 1$. One way of seeing this is as follows: By considering $p$-adic families through $g_\alpha$ and $g_{1/\beta}^*$, Hida constructed a $p$-adic $L$-function attached to the representation space $\End(V_g)$, which will be denoted by $L_p(g,g^*,s)$, that is analytic at $s = 1$ \cite{Hi}. Moreover, the work of Dasgupta on the factorization of $p$-adic $L$-functions implies that 
\[
L_p(g,g^*,s) = \zeta_p(s)L_p(\ad^0(g_\alpha),s),
\]
where $\zeta_p(s)$ is the $p$-adic zeta function \cite{Da}. Since $\zeta_p(s)$ has a simple pole at $s = 1$ it follows that $L_p(\ad^0(g_\alpha),1) = 0$. 
It is thus natural  to define the analytic $\mathcal L$-invariant of $\ad^0(g_\alpha)$ as
\begin{equation}\label{def-Lan}
\mathcal L^{\mathrm{an}}(\ad^0(g_\alpha)):= L_p'(\ad^0(g_\alpha),1),
\end{equation} 
where $L_p'(\ad^0(g_\alpha),s)$ is the the first derivative of $L_p(\ad^0(g_\alpha),s)$ with respect to $s$.

Enlarge $L$ if necessary, so that it contains 
all Fourier coefficients of $g_\alpha$. 
We assume from now on that $L$ may be embedded in $\Q_p$. This hypothesis is not substantial and one could easily dispense with it by working throughout in $L\otimes_{\Q} \Q_p$, at the cost of making the notations slightly more heavy and adding no actual new ideas in the proof. 

Consider the following hypothesis on $g$:
\begin{enumerate}
	\item[(A1)] $g$ is $p$-distinguished, i.e.\,$\alpha \ne \beta \, (\mathrm{mod} \, p)$,
	
	\item[(A2)] $\varrho_g$ is not induced from a character of a real quadratic field in which $p$ splits and
	
	\item[(A3)] the reduction of $\varrho_g$ mod $p$ is irreducible.
	
\end{enumerate}

Assuming Hypothesis (A1), (A2) and (A3), Rivero and the second author computed $\mathcal L^{\mathrm{an}}(\ad^0(g_\alpha))$ as an element of
 $\Q_p^\times/L^\times$ in terms $p$-adic logarithms of global units and $p$-units of $H$. This expression can be found in Theorem A' of \cite{RR1} and will be recalled in \S \ref{SectionLan}.

Now we look at the algebraic side of the theory. 
Assuming that $g$ is $p$-regular (i.e. $\alpha \neq \beta$) and Hypothesis (A2) stated above, it is possible to describe the existence of an exceptional zero in terms of the representation space $V = \ad^0(g) \otimes \Q_p$. As we will see, this is mainly because $V$ contains a line where $G_{\Q_p}$ acts trivially. Moreover, following Greenberg \cite{Gr} one can attach an algebraic $\mathcal L$-invariant, attached to $V$ and $\alpha$, arising from a certain non-zero global cohomology class of $H^1(\Q,V)$. The definition is in terms of universal norms of a $\Z_p$-extension of $\Q_p$ determined by this class.  We will call this invariant Greenberg's $\mathcal L$-invariant and we will denote it by $\mathcal L^{\mathrm{Gr}}(\ad^0(g_\alpha))$.



As we state below explicitly, the main theorem of this note is that, under (A1-2-3) and a minor additional assumption, the above two $\cL$-invariants are equal in $\Q_p^\times \,(\mbox{mod } L^\times)$. The hint that motivated us to look into this question was an old calculation of Greenberg  in \cite{Gr2}, that was explained by the third author to the other two during a visit to Barcelona.  In the very particular case where $g \in S_1(23, \eta)$ is the theta series of an unramified Hilbert class character of $K = \Q(\sqrt{-23})$ and $p = 23$, Greenberg defined and computed an invariant associated to $g$ which remained largely mysterious and adhoc. This note was born with the aim of placing this calculation into a more general and conceptual framework.

The formula Greenberg obtained was the following. Let $H$ be the Hilbert class field of $K$, which is the splitting field of $X^3 - X + 1$. Note that $\Gal(H/\Q) \simeq S_3$ and let $r$ be a generator of $\Gal(H/K)$. Let $\fp_1$ be the prime of $H$ above $p$ singled out by the embedding $\bar{\Q} \subset \bar{\Q}_p$ and let $H_1$ be the subfield of $H$ fixed by the corresponding decomposition subgroup at $\fp_1$. We have that $H_1$ is a cubic extension with one real embedding and a pair of complex embeddings. Let $\varepsilon \in \cO_{H_1}^\times$ be a fundamental unit of $H_1$ and let $\pi$ be a generator of the prime ideal of $H_1$ below $\fp_1$. Then, in \cite[page 230]{Gr2}, Greenberg computed an universal norm arising from his study of the Selmer group of a Galois representation naturally associated to $g$, with the following explicit output (the notation on the left-hand side is ours):
\begin{equation}\label{LGr_p=23}
\mathcal L^{\mathrm{Gr}}(\ad^0(g_\alpha)) = 3\left(\log_p(\pi) - \frac{\log_p(r(\pi)/r^2(\pi))}{\log_p(r(\varepsilon)/r^2(\varepsilon))}\log_p(\varepsilon)\right).
\end{equation}

Moreover, this calculation relates $\mathcal L^{\mathrm{Gr}}(\ad^0(g))$ with the characteristic ideal of the Pontryagin dual of a Selmer group attached to a concrete filtration of $V$ (which depends on $\alpha$). This is also explained in Remark 4.3 of the recent paper  \cite{GV} of Greenberg and the third author and, for a different setting, in Proposition 4 of \cite{Gr}.

Even though the case mentioned above is not considered in our setting, the shape of the expression for $\mathcal L^{\mathrm{an}}(\ad^0(g_\alpha))$ of \cite[Theorem A' ]{RR1} is strikingly similar \eqref{LGr_p=23}. This fact, together with the theoretical evidence that $L_p(\ad^0(g_\alpha),s)$ should be equal to the characteristic ideal mentioned above, suggested that the equality of this two $\mathcal L$-invariants should hold in greater generality. 

\begin{theorem}
	Assume that $g$ satisfies Hypothesis (A1-2-3) mentioned above. In addition, if $g$ is an exotic form assume that $\alpha \neq -\beta$. Then,
	\[
	\mathcal L^{\mathrm{Gr}}(\ad^0(g_\alpha)) \equiv \mathcal L^{\mathrm{an}}(\ad^0(g_\alpha)) \quad (\mathrm{mod} \, L^\times).
	\]
\end{theorem}

The strategy of the proof consists in extending Greenberg's calculation presented above: we compute $\mathcal L^{\mathrm{Gr}}(\ad^0(g_\alpha))$ modulo $\Q^\times$ in terms of global units and $p$-units of $H$. Then, we compare the result with the formula given in Theorem A' of \cite{RR1} in order to prove the theorem.

\vspace{0.15cm}
{\bf Acknowledgments.}  This work has received funding from the European Research Council (ERC) under the European Union's Horizon 2020 research and innovation programme (grant agreement No 682152). The first and the third authors wish to thank Universitat Politècnica de Catalunya and the people from the Number Theory Group of Barcelona for their hospitality during their stay in Barcelona, where part of this research was conducted. The first author expresses his thanks especially to Francesca Gatti and Óscar Rivero, for their willingness to answer many of the questions he had during the preparation of this paper.

\section{The analytic $\mathcal L$-invariant}\label{SectionLan}
Consider the notation given in \S\ref{SectionIntroduction}. We introduce the necessary tools to state Theorem A' of \cite{RR1}. This theorem allows us to express $\mathcal L^{\mathrm{an}}(\ad^0(g_\alpha))$ in terms of global units and $p$-units of $H$ and it will be used later to prove our main result. 

Note that, as $G_{\Q_p}$-modules, we have
\begin{equation}\label{G_Qpdecomposition}
\ad^0(g) = L \oplus L^{\alpha/\beta} \oplus L^{\beta/\alpha},
\end{equation}
where each line is characterized by the property that the arithmetic Frobenius $\Frob_p \in G_{\Q_p}$ acts on it with eigenvalue $1$, $\alpha/\beta$ and $\beta/\alpha$, respectively. Recall that this decomposition is canonical even when $\alpha/\beta = \beta/\alpha$ (see \S2 of \cite{BD} for a detailed explanation of the decomposition stated in \eqref{G_Qpdecomposition}).

\begin{definition}\label{defe}
	Following the decomposition given in \eqref{G_Qpdecomposition}, let $e_1, e_{\alpha/\beta}, e_{\beta/\alpha}$ be generators of the lines $L,\ L^{\alpha/\beta},\ L^{\beta/\alpha}$ respectively. Clearly $\{e_1, e_{\alpha/\beta}, e_{\beta/\alpha}\}$ is a basis of $\ad^0(g)$.
\end{definition}

Now we define the global units and $p$-units that will appear in the expression of $\mathcal L^{\mathrm{an}}(\ad^0(g_\alpha))$. For that we neeed to study the space
\[
\Hom_{G_\Q}(\ad^0(g), \cO_H[1/p]^\times/p^\Z \otimes L).
\] 
By Minkowski's proof of Dirichlet's unit theorem we have that
\[
\dim_{L}\Hom_{G_\Q}(\ad^0(g),\cO_H^\times \otimes L) = 1, \quad \dim_{L}\Hom_{G_\Q}(\ad^0(g), \cO_H[1/p]^\times/p^\Z \otimes L) = 2.
\]
Fix a generator $u$ of $\Hom_{G_\Q}(\ad^0(g),\cO_H^\times \otimes L)$ and $v \in \Hom_{G_\Q}(\ad^0(g), \cO_H[1/p]^\times/p^\Z \otimes L)$ such that $\{u,v\}$ is a basis of $\Hom_{G_\Q}(\ad^0(g), \cO_H[1/p]^\times/p^\Z \otimes L)$. 

\begin{definition}\label{definitionuv}
	Following the notation above, define
	\[
	u_{1} = u(e_1), u_{\alpha/\beta} = u(e_{\alpha/\beta}), u_{\beta/\alpha} = u(e_{\beta/\alpha}),
	v_{1} = v(e_1), v_{\alpha/\beta} = v(e_{\alpha/\beta}), v_{\beta/\alpha} = v(e_{\beta/\alpha}).
	\]
\end{definition}

By construction we have $u_1, v_1 \in \Q_p^\times$ and, 
\[
\Frob_p(u_{\alpha/\beta}) = \frac{\alpha}{\beta} u_{\alpha/\beta}, \quad \Frob_p(v_{\alpha/\beta}) = \frac{\alpha}{\beta} v_{\alpha/\beta}, \quad \Frob_p(u_{\beta/\alpha}) = \frac{\beta}{\alpha} u_{\beta/\alpha}, \quad \Frob_p(v_{\beta/\alpha}) = \frac{\beta}{\alpha} v_{\beta/\alpha}.
\]

Let
$$
\log_p: H_{\fp_1}^\times \otimes L \lra H_{\fp_1} \otimes L
$$
denote the usual $p$-adic logarithm. Here, $\fp_1$ is the prime of $H$ singled out by the embeding $\bar{\Q} \subset \bar{\Q}_p$ and $H_{\fp_1}$ is the completion of $H$ at this prime.

Recall the definition of the analytic $\cL$-invariant.

\begin{definition}
	Define the analytic $\mathcal L$-invariant attached to $g$ and $\alpha$ by
	\[
	\mathcal L^{\mathrm{an}}(\ad^0(g_\alpha)):= L_p'(\ad^0(g_\alpha),1),
	\]
	where $L_p'(\ad^0(g_\alpha),1)$ is the derivative of Hida-Schmidt's $p$-adic $L$-function associated to the adjoint of $g_\alpha$ at $s=1$.
\end{definition}

The following result was proved as Theorem A' of \cite{RR1}.
\begin{theorem}\label{main1'}
	Assume that $g$ satsifies the hypothesis (A1-2-3) stated in the introduction. Following the notation above, we have
	\[
	\mathcal L^{\mathrm{an}}(\ad^0(g_\alpha))  \equiv  \frac{ \log_p(u_1)\log_p(v_{\beta/\alpha}) - \log_p(u_{\beta/\alpha})\log_p(v_1)}{\log_p(u_{\beta/\alpha})} \quad (\mathrm{mod} \, L^\times).
	\]
\end{theorem}
\begin{proof}
	See Theorem A' of \cite{RR1}.
\end{proof}
\begin{remark}
	The elements $e_1, e_{\alpha/\beta}, e_{\beta/\alpha}$ are well defined up to a constant. Similarly, if $\{u, v\}$ is a pair as above any other possible pair is of the form $\{x_uu, x_v u + y_v v\}$ with $x_u, y_v \in L^\times$ and $x_v \in L$. From here it is a computation to check that the right hand side of the expression for $\mathcal L^{\mathrm{an}}(\ad^0(g_\alpha)) $ is well defined modulo $L^\times$.
\end{remark}

\section{Greenberg's $\mathcal L$-invariant}\label{SectionLGr}
Consider the same notation as in \S \ref{SectionIntroduction}. Following  \cite{Gr}, in this section we define Greenberg's $\mathcal L$-invariant $\mathcal L^\mathrm{Gr}(\ad^0(g_\alpha))$ attached to $g \in S_1(N,\eta)$, $\alpha$ and $p$, a rational prime such that $p \nmid N$. We then relate $\mathcal L^\mathrm{Gr}(\ad^0(g_\alpha))$ to a global $\Z_p$-extension of $H$ which will be the crucial ingredient in later sections for computing it.

Recall that $g$ has associated the Artin representation $\varrho_g$ and its adjoint is
\[
\varrho_{\ad^0(g)}: \Gal(H/\Q) \hookrightarrow \GL(\ad^0(g)) \simeq \GL_3(L),
\]
where $H$ is a finite Galois extension of $\Q$ and $L$ is a sufficiently large finite extension of $\Q$ that we assumed could be embedded in $\Q_p$. Note that since $p \nmid N$ we have that $p$ is unramified in $H$. We begin stating the hypothesis we make about $g$ and introducing some notation. Suppose that:

\begin{enumerate}
	
	\item[(B1)] $g$ is $p$-regular, i.e. $\alpha/\beta \neq 1$, and
	
	\item[(B2)] $\varrho_g$ is not induced from a character of a real quadratic field in which $p$ splits.
	
\end{enumerate}
These hypothesis are less restrictive than Hypothesis (A1-2-3) made in the introduction in order to apply Theorem A' of \cite{RR1}.
 
Let $V$ be a $p$-adic realization of $\ad^0(g)$, i.e. $V = \ad^0(g) \otimes \Q_p$. Let $\fp_1$ be the prime of $H$ above $p$ singled out by the embedding $\bar{\Q} \subset \bar{\Q}_p$. Let $\Sigma = \{p , \infty \} \cup \mathrm{Ram}(V)$, where $\mathrm{Ram}(V)$ is the set of rational primes that ramify in $\ad^0(g)$ and let $\Q_\Sigma$ be the maximal extension of $\Q$ unramified outside $\Sigma$ (note that $H \subset \Q_\Sigma$). Let $\Delta = \Gal(H/\Q)$. 

By \eqref{G_Qpdecomposition}, we have that as $G_{\Q_p}$-modules
\begin{equation}\label{decVG_Qp}
V = \Q_p \oplus \Q_p^{\alpha/\beta} \oplus \Q_p^{\beta/\alpha},
\end{equation}
where these lines are determined as in \eqref{G_Qpdecomposition}. Recall that this decomposition is canonical even when $\alpha/\beta = \beta/\alpha$. In order to define $\mathcal L^\mathrm{Gr}(\ad^0(g_\alpha))$ fix the following $G_{\Q_p}$-equivariant line of $V$.
\begin{definition}
	Following the notation above, let $F^+V = \Q_p^{\beta/\alpha}$. Note that $V/F^+V = \Q_p \oplus \Q_p^{\alpha/\beta}$ and we have a natural injection $H^1(\Q_p,\Q_p) \subset H^1(\Q_p,V/F^+V)$.
\end{definition}
\begin{remark}
	The choice $F^+V = \Q_p^{\alpha/\beta}$ would correspond to $\mathcal L^{\mathrm{Gr}}(\ad^0(g_\beta))$.
\end{remark}

Before starting the procedure of defining $\mathcal{L}^\mathrm{Gr}(\ad^0(g_\alpha))$ we explain its outline. From the study of the $\Q_p$-vector space $H^1(\Q_\Sigma/\Q,V)$ and of the natural map
\begin{equation}\label{lambda}
\lambda: H^1(\Q_\Sigma/\Q,V) \to H^1(\Q_p, V/F^+V)
\end{equation}
we will identify a $1$-dimensional subspace of $H^1(\Q_p, \Q_p) = \Hom(G_{\Q_p}, \Q_p)$ (here, and from now on, we are considering continuous homomorphisms). Then, using class field theory, we will express this $1$-dimensional subspace as $\Hom(\Gal(H_\infty/\Q_p), \Q_p)$, where $H_\infty$ is a $\Z_p$-extension of $\Q_p$. Finally, $\mathcal L^\mathrm{Gr}(\ad^0(g_\alpha))$ is defined in terms of $H_\infty$ (in fact, it is defined in terms of some universal norm of $H_\infty$).

\begin{remark}
	As we will see, the fact that $V/F^+V$ contains a line where $G_{\Q_p}$ acts trivially is essential in order to determine the $1$-dimensional subspace of $\Hom(G_{\Q_p}, \Q_p)$ and define the $\mathcal L$-invariant. In other words, it is the main reason why we have an exceptional zero.
\end{remark}

\subsection{Identification of the $1$-dimensional subspace of $\Hom(G_{\Q_p},\Q_p)$}
We begin studying the $\Q_p$-vector space $H^1(\Q_\Sigma/\Q,V)$: we give a description of it using class field theory and we compute its dimension. For that, we introduce the following notation. Let $\fp_1, \dots, \fp_n$ be the primes of $H$ above $p$. For every $i$, let $H_{\fp_i}$ be the completion of $H$ at $\fp_i$, let $\Delta_{\fp_i} = \Gal(H_{\fp_i}/\Q_p) \subset \Delta$ be the corresponding decomposition subgroup and denote by $U_{\fp_i}^1 \subset H_{\fp_i}^\times$ the subgroup of principal units. Note that $\Delta_{\fp_i}$ acts on $U_{\fp_i, \Q_p}^1$ as the regular representation , where we denote $U_{\fp_i, \Q_p}^1 = U_{\fp_i}^1 \otimes_{\Z_p} \Q_p$. This can be proven using the $\Delta_{\fp_i}$-equivariant isomorphism $\log_p: U_{\fp_i, \Q_p}^1 \to H_{\fp_i}$ given by the $p$-adic logarithm map. By identifying $\prod_{i = 1}^n U_{\fp_i,\Q_p}^1$ with $\Ind_{\Delta_{\fp_1}}^{\Delta}(U_{\fp_1,\Q_p}^1)$ in the natural way, we see that $\Delta$ acts on $\prod_{i = 1}^n U_{\fp_i,\Q_p}^1$ as the regular representation (since $\Ind_{\Delta_{\fp_1}}^\Delta(\Reg_{\Delta_{\fp_1}}) = \Reg_{\Delta}$). Let $\cO_H^1$ be the global units of $H$ that are congruent to $1$ modulo $\fp_i$ for every $i$. The diagonal embedding 
\[
\cO_H^1 \to \prod_{i = 1}U_{\fp_i}^1
\]
is $\Delta$-equivariant. This map can be extended to $\cO_H^1 \otimes_\Z \Z_p \to \prod_{i = 1}^n U_{\fp_i}^1$ and we denote the image of this map by $\overline{\cO_{H}^1}$. Furthermore, tensoring these objects by $\Q_p$ (over $\Z_p$) we obtain the $\Delta$-equivariant map $\cO_H^1 \otimes_\Z \Q_p  \to \prod_{i = 1}^n U_{\fp_i,\Q_p}^1$. Let $\cO_{H,\Q_p}^1 = \cO_H^1 \otimes_\Z \Q_p$ and let  $\overline{\cO_{H,\Q_p}^1}$ be the image of $\cO_{H,\Q_p}^1$ by the last map. Note the following result from class field theory.

\begin{lemma}\label{cft}
	Let $\tilde{H}$ be the compositum of all $\Z_p$-extensions of $H$. The global Artin map gives a $\Delta$-equivariant map
	\[
	\prod_{i = 1}^n U_{\fp_i}^1 / \overline{\cO_{H}^1} \to \Gal(\tilde{H}/H)
	\]
	which is injective and has finite cokernel.
\end{lemma}
The following result is already explained in \cite{BD} and \cite{GV}, but we present it here since it will be useful in the next sections in order to compute Greenberg's $\mathcal L$-invariant.
\begin{proposition}\label{dimLHS}
	The vector spaces $H^1(\Q_\Sigma/\Q, V)$ and $\Hom_{\Delta}(\prod_{i = 1}^n U_{\fp_i}^1/\overline{\cO_H^1}, V)$ are naturally isomorphic. Moreover, they have dimension $2$ over $\Q_p$.
\end{proposition}
\begin{proof}
	The inflation-restriction exact sequence gives the isomorphism
	\[
	H^1(\Q_\Sigma/\Q,V) \xrightarrow{\sim} H^1(\Q_\Sigma/H,V)^\Delta = \Hom_\Delta(\Gal(\Q_\Sigma/H),V)
	\]
	since the kernel and cokernel of this map are cohomology spaces of a finite group with values in a characteristic zero representation (this argument will be used from now on without explanation). Using that $V$ is a $\Q_p$-vector space we have
	\[
	\Hom_\Delta(\Gal(\Q_\Sigma/H),V) = \Hom_\Delta(\Gal(M/H), V), 
	\]
	where $M$ is the maximal abelian pro-$p$ extension contained in $\Q_\Sigma$. If we let $\tilde{H}$ be the compositum of all $\Z_p$-extensions of $H$, 
	then $H \subset \tilde{H} \subset M$ and $M/\tilde{H}$ is a finite extension. Therefore, using Lemma \ref{cft} for the last isomorphism
	\[
	\Hom_{\Delta}(\Gal(M/H), V) \simeq \Hom_{\Delta}(\Gal(\tilde{H}/H),V) \simeq \Hom_{\Delta}\left(\prod_{i = 1}^n U_{\fp_i}^1/\overline{\cO_H^1}, V\right).
	\]
	This concludes the proof of the first part of the proposition. To compute the dimension of $H^1(\Q_\Sigma/\Q, V)$ we will compute the dimension of $\Hom_{\Delta}\left(\prod_{i = 1}^n U_{\fp_i,\Q_p}^1/\overline{\cO_{H,\Q_p}^1}, V\right)$. This can be done by finding the multiplicity of the irreducible components of $V$ in $\prod_{i = 1}^n U_{\fp_i,\Q_p}^1/\overline{\cO_{H,\Q_p}^1}$ and applying Schur's lemma.
	
	We make some general observations before proceeding by case analysis according to the decomposition of $V$ into irreducible components. Let $\pi$ be an irreducible representation of $\Delta$ realizable over $L$. Denote by $\dim \pi$ the degree of this representation and by $\dim \pi^+$ the dimension of the subspace of $\pi$ fixed by complex conjugation. We have:
	\begin{enumerate}
		\item[(i)] The multiplicity of $\pi$ in $\prod_{i = 1}^n U_{\fp_i,\Q_p}^1$ is $\dim \pi$.
		\item[(ii)] If $\pi$ is not the trivial representation, by Minkowski's proof of Dirichlet's unit theorem, the multiplicity of $\pi$ in $\cO_H^1 \otimes L$ is $\dim \pi^+$.
		\item[(iii)] If $\dim \pi^+ = 0,1$, then the multiplicity of $\pi$ in $\overline{\cO_{H,\Q_p}^1}$ is $\dim \pi^+$ (it is clear when $\dim \pi^+ = 0$ and it is Corollary 2.7 of \cite{GV} when $\dim \pi^+ = 1$).
		\item[(iv)] Since $\varrho_g$ is an odd representation, $\varrho_{\ad^0(g)}(\tau)$ has eigenvalues $1, -1, -1$. 
	\end{enumerate}
	Now we do the case analysis:
	\begin{itemize}
		\item If $V$ is irreducible, i.e. $g$ is an exotic form, the multiplicity of $\ad^0(g)$ in $\prod_{i = 1}^3 U_{\fp_i, \Q_p}^1$ is $3$ and the multiplicity in $\overline{\cO_{H,\Q_p}^1}$ is $1$. Thus, the multiplicity of $\ad^0(g)$ in $\prod_{i = 1}^n U_{\fp_i,\Q_p}^1/\overline{\cO_{H,\Q_p}^1}$ is $2$.
		\item If $V$ decomposes as the sum of two irreducible representations, then there exists a quadratic field $K$ and a $1$-dimensional character $\psi: G_K \to \C^\times$ such that
		\[
		V = \chi_K \oplus \Ind_{G_K}^{G_\Q}(\psi),
		\]
		where $\chi_K: G_{\Q} \to \Q^\times$ is the character of order two attached to $K$. Proceeding in a similar way as above, we conclude that both $\chi_K$ and $\Ind_{G_K}^{G_\Q}(\psi)$ have multiplicity $1$ in $\prod_{i = 1}^n U_{\fp_i,\Q_p}^1/\overline{\cO_{H,\Q_p}^1}$ if $K$ is imaginary quadratic and $\chi_K$ has multiplicity $0$ and $\Ind_{G_K}^{G_\Q}(\psi)$ has multiplicity $2$ when $K$ is real quadratic.
		\item If $V$ decomposes as the sum of three $1$-dimensional representations there exists a real quadratic field $K_1$ and two imaginary quadratic fields $K_2$ and $K_3$ such that
		\[
		V = \chi_{K_1} \oplus \chi_{K_2} \oplus \chi_{K_3},
		\] 
		where $\chi_{K_i}$ is the character of order two attached to $K_i$ for every $i$. Then, the multiplicity of $\chi_{K_i}$ in $\prod_{i = 1}^n U_{\fp_i,\Q_p}^1$ is $0$ for $i = 1$ and $1$ for $i = 2,3$.
	\end{itemize}
\end{proof}
The next two propositions will be useful to study the map $\lambda$ apprearing in \eqref{lambda} and its image. The second one is essential for us in order to define Greenberg's $\mathcal L$-invariant and it is the only place where we use Hypothesis (B2). To simplify the notation denote

\[
W = \ker\left(H^1(\Q_p, V/F^+V) \to H^1(I_p,V/F^+V)\right) \subset H^1(\Q_p, V/F^+V).
\]

\begin{proposition}\label{dimRHS}
	We have $\dim_{\Q_p} H^1(\Q_p,V/F^+V) = 3$ and $\dim_{\Q_p} W = 1$. Moreover, $W \subset H^1(\Q_p,\Q_p)$ as subspaces of $H^1(\Q_p,V/F^+V)$.
\end{proposition}
\begin{proof}
	We will use that $\mathrm{ad}^0(g)$ is an Artin representation. Since the action of $G_\Q$ on $V$ factors through $\Delta = \Gal(H/\Q)$, the action of $G_{\Q_p}$ on $V$ factors through the decomposition subgroup $\Delta_{\fp_1}$. Using the inflation-restriction exact sequence, we deduce the following isomorphism
	\[
	H^1(\Q_p, V/F^+V) \xrightarrow{\sim} \Hom_{\Delta_{\fp_1}}(G_{H_{\fp_1}}, V/F^+V).
	\]
	We use local class field theory to rewrite it. Denote by $\cO_{H_{\fp_1}}^\times$ the group of local units of $H_{\fp_1}$. Since $p$ is unramified in $H$, we have that $p$ is a uniformizer of $H_{\fp_1}$. Therefore $H_{\fp_1}^\times = p^\Z \times \cO_{H_{\fp_1}}^\times$. Using the local reciprocity map, we get a natural isomorphism $\Hom_{\Delta_{\fp_1}}(G_{H_{\fp_1}}, V/F^+V) \simeq \Hom_{\Delta_{\fp_1}}(p^\Z \times \cO_{H_{\fp_1}}^\times, V/F^+V)$. Composing these two isomorphisms we have
	\begin{equation}\label{rescftiso}
	\mu: H^1(\Q_p, V/F^+V) \xrightarrow{\sim} \Hom_{\Delta_{\fp_1}}(p^\Z \times \cO_{H_{\fp_1}}^\times, V/F^+V).
	\end{equation}
	In order to prove the proposition we will study the space $\Hom_{\Delta_{\fp_1}}(p^\Z \times \cO_{H_{\fp_1}}^\times, V/F^+V)$. Recall that $\Delta_{\fp_1}$ acts on $U_{\fp_1,\Q_p}^1$ as the regular representation. From here, we deduce that $\left(p^\Z \times \cO_{H_{\fp_1}}^\times\right) \hat{\otimes} \Q_p$ is isomorphic to the direct sum of the trivial representation and the regular representation of $\Delta_{\fp_1}$, where we are considering the completed tensor product over $\Z$. On the other hand, $V/F^+V$ is the direct sum of the trivial representation and a $1$-dimensional non-trivial representation (we are using Hypothesis (B1)). By computing the multiplicity of these two irreducible representations in $\left(p^\Z \times \cO_{H_{\fp_1}}^\times\right)  \hat{\otimes} \Q_p$ we deduce the first equality of the proposition. Now, note that inflation-restriction and class field theory give the injection $H^1(I_p,V/F^+V) \hookrightarrow \Hom(\cO_{H_{\fp_1}}^\times, V/F^+V)$. From here, $f \in W$ if and only if $\mu(f)(\cO_{H_{\fp_1}}^\times) = 0$. Hence, $\mu(W) = \Hom_{\Delta_{\fp_1}}(p^\Z, \Q_p) \subset \Hom_{\Delta_{\fp_1}}(p^\Z \times \cO_{H_{\fp_1}}^\times , \Q_p)$ and the last two statements follow.
	
\end{proof}

\begin{proposition}\label{propisomorphism}
	The natural map
	\begin{equation}\label{isomorphism}
	H^1(\Q_\Sigma/\Q,V) \to H^1(\Q_p,V/F^+V)/W
	\end{equation}
	is an isomorphism. In particular, the map $\lambda$ in \eqref{lambda} is injective.
\end{proposition}
\begin{proof}
	By Proposition \ref{dimLHS} and Proposition \ref{dimRHS}, $H^1(\Q_\Sigma/\Q,V)$ has the same dimension than $H^1(\Q_p, V/F^+V)/W$. Therefore, it is enough to see that \eqref{isomorphism} is injective, i.e. we need to prove that the kernel of the natural map
	\begin{equation}\label{conditionSelmer}
	H^1(\Q_\Sigma/\Q,V) \to H^1(I_p, V/F^+V) = H^1(I_p, \Q_p) \oplus H^1(I_p, \Q_p^{\alpha/\beta})
	\end{equation}
	is trivial (where we used \eqref{decVG_Qp} to decompose $V/F^+V$). For that we will use Theorem 2.2 of \cite{BD}. By Hypothesis (B2) and Lemma 2.3 of \cite{BD}, Theorem 2.2 of \cite{BD} (and the comment below the statement of this theorem) imply that the natural map 
	\[
	H^1(\Q_\Sigma/\Q,V) \to H^1(I_p,\Q_p) \oplus H^1(\Q_p, \Q_p^{\alpha/\beta})
	\]
	is injective. Therefore, in order to conclude it is enough to see that the restriction map
	\[
	H^1(\Q_p, \Q_p^{\alpha/\beta}) \to H^1(I_p, \Q_p^{\alpha/\beta})
	\]
	is injective as well. For that we proceed as in the previous proof: by local class field theory and the inflation-restriction exact sequence we get the following commutative diagram
	\[
	\begin{tikzcd}
	H^1(\Q_p,\Q_p^{\alpha/\beta}) \arrow[r,""] \arrow[d, " "]
	& H^1(I_p,\Q_p^{\alpha/\beta}) \arrow[d, ""] \\
	\Hom_{\Delta_{\fp_1}}(p^\Z \times \cO_{H_{\fp_1}}^\times, \Q_p^{\alpha/\beta}) \arrow[r, ""]
	& \Hom(\cO_{H_{\fp_1}}^\times, \Q_p^{\alpha/\beta}),
	\end{tikzcd}
	\]
	where the left vertical map is an isomorphism and the horizontal maps are the usual restrictions. Since $\alpha/\beta \neq 1$, we have $\Hom_{\Delta_{\fp_1}}(p^\Z \times \cO_{H_{\fp_1}}^\times, \Q_p^{\alpha/\beta}) = \Hom_{\Delta_{\fp_1}}(\cO_{H_{\fp_1}}^\times, \Q_p^{\alpha/\beta})$ so the injectivity of the top horizontal map follows.
	
\end{proof}

\begin{remark}\label{interpretationSelmer}
	Proving that \eqref{conditionSelmer} is injective is essentially proving that a Selmer group attached to $p$ and to the filtration $V \supset F^+V$ is finite. A definition of this Selmer group can be found in \cite{GV}.
\end{remark}

We use this isomorphism to identify a $1$-dimensional subspace of $H^1(\Q_\Sigma/\Q,V)$ such that its image by $\lambda$ gives the desired subspace of $H^1(\Q_p,\Q_p) = \Hom(G_{\Q_p}, \Q_p)$. Then, we state a property of it that will be essential in order to define $\mathcal L^\mathrm{Gr}(\ad^0(g_\alpha))$.

\begin{definition}\label{defC}
Define $C \subset H^1(\Q_\Sigma/\Q,V)$ to be the preimage of $H^1(\Q_p,\Q_p)/W$ by the isomorphism \eqref{isomorphism}.
\end{definition}
 
\begin{corollary}\label{lambda(C)}
	The image of $C \subset H^1(\Q_\Sigma/\Q, V)$ by the map 
	\[
	\lambda: H^1(\Q_\Sigma/\Q,V) \to H^1(\Q_p, V/F^+V)
	\]
	satisfies
	\begin{enumerate}
		\item $\lambda(C)$ is a $1$-dimensional subspace of $H^1(\Q_p, \Q_p) = \Hom(G_{\Q_p},\Q_p)$,
		\item $\lambda(C) \cap W = 0$.
	\end{enumerate}
\end{corollary}
\begin{remark}\label{lambda(C)=intersection}
	Note that we can also describe this space as $\lambda(C) = \lambda(H^1(\Q_\Sigma/\Q,V)) \cap H^1(\Q_p, \Q_p)$.
\end{remark}

\subsection{Definition of $\mathcal L^\mathrm{Gr}(\ad^0(g_\alpha))$ in terms of universal norms}

\begin{definition}\label{Zpextn}
	Let $H_\infty$ be the unique $\Z_p$-extension of $\Q_p$ such that $\lambda(C) = \Hom(\Gal(H_\infty/\Q_p),\Q_p)$. By Corollary \ref{lambda(C)} (2), $H_\infty$ is not the unramified $\Z_p$-extension of $\Q_p$.
\end{definition}
Let $H_n$ be the subfield of $H_\infty$ of degree $p^n$ over $\Q_p$. The definition of $\mathcal L^\mathrm{Gr}(\ad^0(g_\alpha))$ will be done in terms of elements of the group $\mathrm{UnivNorm}(H_\infty) := \cap \Norm_{H_n/\Q_p}(H_n^\times)$, the so called group of universal norms of $H_\infty$. The following application of local class field theory tells us the structure of this group. Here, and from now on, denote by $\mu_{p-1}$ the group of $(p-1)$th roots of unity.
\begin{lemma}\label{structUnivNorm}
	The group $\mathrm{UnivNorm}(H_\infty)$ has the following structure
	\[
	\mathrm{UnivNorm}(H_\infty) = \mu_{p-1} \times r^\Z,
	\]
	where $r = p^k s$ for some $k \geq 1$ and $s \in 1+ p\Z_p$.
\end{lemma}
\begin{proof}
	Note that $\mathrm{UnivNorm}(H_\infty)$ is the kernel of the continuous morphism
	\[
	\Q_p^\times \rightarrow \Gal(\Q_p^{\mathrm{ab}}/\Q_p) \twoheadrightarrow \Gal(H_\infty/\Q_p), 
	\]
	where the first map is the local Artin map and the second one is the natural projection. Then, we have to use that $\Q_p^\times \simeq \mu_{p-1} \times \Z_p \times p^\Z$, $\Gal(H_\infty/\Q_p) \simeq \Z_p$ and the fact that $H_\infty$ is not the unramified extension of $\Q_p$.
	
\end{proof}
\begin{remark}
	In fact, it can be seen that $k$ as above is equal to $p^t$ for some $t \geq 0$ as observed in Example 4 of \cite{Gr2} but we will not need this.
\end{remark}

And we can now define Greenberg's $\mathcal L$-invariant.
\begin{definition}\label{defLGr}
	Let $q \in \mathrm{UnivNorm}(H_\infty)$ such that $q \not \in \mu_{p-1}$. Define 
	\[
	\mathcal L^\mathrm{Gr}(\ad^0(g_\alpha)) = \frac{\log_p(q)}{\ord_p(q)},
	\]
	where $\log_p$ is the usual extension of the $p$-adic logarithm satisfying $\log_p(p) = 0$.
\end{definition}

\begin{remark}
	By Lemma \ref{structUnivNorm} we see that $\mathcal L^\mathrm{Gr}(\ad^0(g_\alpha))$ is well defined: for any $q$ as in the definition we have $\ord_p(q) \neq 0$ and the definition does not depend on the choice of $q$.
\end{remark}

\begin{remark}
	We have that $\mathcal L^{\mathrm{Gr}}(\ad^0(g_\alpha)) = 0$ if and only if $\mathrm{UnivNorm}(H_\infty)= \mu_{p-1} \times p^\Z$. In other words, $\mathcal L^{\mathrm{Gr}}(\ad^0(g_\alpha)) = 0$ if and only if $H_\infty$ is the cyclotomic $\Z_p$-extension of $\Q_p$.  
\end{remark}

\begin{remark}\label{subspaceofad}
	Instead of working with $V$, in order to define $\mathcal L^{\mathrm{Gr}}(\ad^0(g_\alpha))$ it would be enough to work with the irreducible subspace of $V$ containing the line where $G_{\Q_p}$ acts trivially. In other words, it would be enough to prove an isomorphism analogous to \eqref{isomorphism} replacing $V$ by this irreducible subspace (and therefore changing $F^+V$ accordingly). For the cases when $g$ is the theta series of a character of $K$ quadratic imaginary or $g$ is an exotic form such that $(\alpha/\beta)^2 \neq 1$ this isomorphism can also be proven using Theorem 1 of \cite{GV} and a statement analogous to our Proposition \ref{dimRHS}. In \cite[Theorem 1]{GV} it is shown that the Selmer group attached to this subrepresentation and $p$ is finite (as we mentioned in Remark \ref{interpretationSelmer}).
\end{remark}

\begin{remark} Let $g$ be the theta series of a character of a real quadratic field $K$ where $p$ splits. Recall that this case was excluded by Hypothesis (B2). The only part where we had to use this hypothesis was to prove the injectivity of \eqref{isomorphism}. In fact, as it is proven in Theorem 2 of \cite{BD}, when $g$ is the theta series of a character of $K$ the map \eqref{isomorphism} is never injective. In order to show this using our previous results, consider the setting of this section and recall that in this case
	\[
	V = \chi_K \oplus \Ind_{G_K}^{G_\Q}(\psi),
	\]
	where $\chi_K: G_{\Q} \to \Q^\times$ is the character of order two attached to $K$ and $\psi$ is a finite order character of $K$. For the sake of contradiction suppose that Proposition \ref{propisomorphism} holds in this case, i.e. suppose that \eqref{isomorphism} is injective. Then, following the discussion done in this section there exists $C \subset H^1(\Q_\Sigma/\Q,V)$ satisfying the conditions stated in Corollary \ref{lambda(C)}. As it is explained in the proof of Proposition \ref{dimLHS}, since $K$ is real quadratic, $H^1(\Q_\Sigma/\Q,V) = H^1(\Q_\Sigma/\Q, \Ind_{G_K}^{G_\Q}(\psi))$ (and $H^1(\Q_\Sigma/\Q, \chi_K) = 0$). On the other hand, the line of $V$ where $G_{\Q_p}$ acts trivially is the line of $V$ where $G_\Q$ acts by $\chi_K$ because $p$ splits in $K$. This is a contradiction with Corollary \ref{lambda(C)} (1), since it says that $\lambda(C)$ is a $1$-dimensional subspace of $H^1(\Q_p,\Q_p)$ but $\Ind_{G_K}^{G_{\Q}}(\psi)$ contains no line where $G_{\Q_p}$-acts trivially.
\end{remark}

\subsection{Towards the explicit computation of $\mathcal L^{\mathrm{Gr}}(\ad^0(g_\alpha))$} We write a universal norm of $H_\infty$ as a product of a fixed $p$-unit of $H$ and a principal unit of $\Z_p$, which is unknown. Then, we use class field theory to give a condition on this unknown in terms of the $\Delta$-representation space $\prod_{i = 1}^n U_{\fp_i,\Q_p}^1$. This will be essential to compute Greenberg's $\mathcal L$-invariant in the next sections. In order to define the fixed $p$-unit, we introduce the following notation. Let $H_1$ be the subfield of $H$ fixed by $\Delta_{\fp_1}$. Let $h$ be its the class number and let $\fp_1'$ be the prime of $H_1$ below $\fp_1$. Note that the completion of $H_1$ at the prime $\fp_1'$ is $\Q_p$. Recall that by Lemma \ref{structUnivNorm}, $\mathrm{UnivNorm}(H_\infty) = \mu_{p-1} \times r^\Z$, where $r = p^k s$ for some $k \geq 1$ and $s \in 1+ p\Z_p$. 

\begin{definition}\label{defpi}
	Following the notation above, fix $\pi' \in H_1$ such that $(\pi') = (\fp_1')^{kh}$. Let $\pi = (\pi')^N$, where $N = \Norm_{H/\Q}(\fp_1)-1$. Note that $\pi$ is congruent to $1$ modulo $\fp_i$ for all $i = 2, \dots, n$.
\end{definition}


\begin{proposition}\label{piepsilona}
	There exists a unique $x \in 1+p\Z_p$ such that $\pi x \in \mathrm{UnivNorm}(H_\infty)$.
\end{proposition}
\begin{proof}
	Viewing $\pi \in \Q_p^\times$ we have that $p^{khN}/\pi \in \Z_p^\times$, where $N = \Norm_{H/\Q}(\fp_1)-1$. Hence, there exists $x \in 1+ p\Z_p$ and $\zeta \in \mu_{p-1}$ such that
	\[
	\zeta\pi x = (p^ks)^{hN}.
	\]
	Since $\zeta \in \mathrm{UnivNorm}(H_\infty)$ we have that $x \pi \in \mathrm{UnivNorm}(H_\infty)$. The uniqueness follows from the structure of the group $\mathrm{UnivNorm}(H_\infty)$. 
\end{proof}

The element $x$ of the previous proposition is the unknown we have to determine in order to compute $\mathcal L^{\mathrm{Gr}}(\ad^0(g_\alpha))$. Since $\pi x \in \mathrm{UnivNorm}(H_\infty)$, the following lemma gives a condition on the value of $x$. 
\begin{lemma}\label{lambda(f)(q)=0}
	Let $q \in \Q_p^\times$ and let $f \in H^1(\Q_\Sigma/\Q, V)$ be a generator of $C$. Then, $q \in \mathrm{UnivNorm}(H_\infty)$ if and only if $\lambda(f)(q) = 0$, where we are using local class field theory to indentify $\Hom(G_{\Q_p},\Q_p) \simeq \Hom(\Q_p^\times,\Q_p)$.
\end{lemma}
\begin{proof}
	From Definition \ref{Zpextn} we have $\lambda(C) = \Hom(\Gal(H_\infty/\Q),\Q_p)$. Then the result follows from local class field theory.
\end{proof}

Finally, the relation between the local and global Artin maps allows us to restate this condition in terms of the following element of $\prod_{i=1}^n U_{\fp_i,\Q_p}^1$.

\begin{definition}\label{defwKQIpinert}
	Let $w = (x, 1/\pi, \dots, 1/\pi) \in \prod_{i = 1}^n U_{\fp_i,\Q_p}^1$, where $x \in 1 + p\Z_p$ is as in Proposition \ref{piepsilona}.
\end{definition}
\begin{remark}
	Note that $1/\pi \in H_1 \subset H$. Hence, we can view $1/\pi \in H_{\fp_i}$ for every $i$ by completing $H$ at the prime $\fp_i$ (since there is a natural inclusion $H \subset H_{\fp_i}$). Moreover, for $i = 2, \dots, n$, $\pi$ is congruent to $1$ modulo $\fp_i$ so we naturally have $1/\pi \in U_{\fp_i,\Q_p}^1$.
\end{remark}

\begin{proposition}\label{f(w)=0KQIpinert}
	Let $f \in H^1(\Q_\Sigma/\Q,V) \simeq \Hom_{\Delta}(\prod_{i = 1}^n U_{\fp_i,\Q_p}^1/\overline{\cO_{H,\Q_p}^1}, V)$ be a generator of $C$. We have
	\[
	f(w) = f(x, 1/\pi, \dots , 1/\pi) = 0.
	\]
\end{proposition}
\begin{proof}
	We have the following commutative diagram
	\[
	\begin{tikzcd}
	H^1(\Q_\Sigma/\Q,V) \arrow[r,"\lambda"] \arrow[d, " "]
	& H^1(\Q_p, V/F^+V) \arrow[d, "\mu"] \\
	\Hom_{\Delta}(\prod_{i = 1}^n U_{\fp_i}^1/\overline{\cO_H^1}, V) \arrow[r, ""]
	& \Hom_{\Delta_{\fp_1}}(H_{\fp_1}^\times, V/F^+V)
	\end{tikzcd}
	\]
	where the vertical maps are isomorphisms given by the restriction maps and class field theory (recall that $\mu$ was given in \eqref{rescftiso}). We proceed to explain in more detail the bottom horizontal map. By Lemma \ref{cft}, we have the isomorphism
	\[
	\Hom_{\Delta}\left(\prod_{i = 1}^n U_{\fp_i}^1/\overline{\cO_H^1}, V\right) \xrightarrow{\sim} \Hom_{\Delta}(\Gal(\tilde{H}/H), V),
	\]
	where $\tilde{H}$ is the compositum of all $\Z_p$-extensions of $H$. If we let $G_{H_{\fp_1}} \subset G_{\Q_p}$ be the decomposition subgroup, we can consider
	\[
	\Hom_{\Delta}(\Gal(\tilde{H}/H), V) \to \Hom_{\Delta_{\fp_1}}(G_{H_{\fp_1}}, V) \to \Hom_{\Delta_{\fp_1}}(G_{H_{\fp_1}},V/F^+V),
	\]
	where the first map is the restriction morphism. Composing these maps and using local class field theory gives the desired map. Now, if $f \in C \subset \Hom_{\Delta}(\prod_{i = 1}^n U_{\fp_i}^1/\overline{\cO_H^1}, V)$ we have that
	\[
	f(w) = \lambda(f)\left((\pi x)^{\#\Delta_{\fp_1}}\right),
	\]
	where the equality takes place in $V/F^+V$. Now, Lemma \ref{lambda(f)(q)=0} shows that $\lambda(f)(\pi x) = 0$ in $V/F^+V$. Since $G_{\Q_p}$ acts trivially on $w$ and hence on $f(w)$ we conclude that $f(w) = 0$ in $V$.
\end{proof}

In the next sections we will study the possible decompositions of $V$ as a sum of irreducible representations. In doing so, we will be able to be more explicit in the description of the subspace $C$ and in the statement of Proposition \ref{f(w)=0KQIpinert}. This will allow us to determine $x$ and compute Greenberg's $\mathcal L$-invariant in terms of global units and $p$-units of $H$.

\section{The case of imaginary quadratic fields in which $p$ splits}\label{SectionIQpsplits}

Consider the notation given in \S\ref{SectionIntroduction} and \S\ref{SectionLGr}.  Consider a $1$-dimensional character $\psi_g: G_K \to \C^\times$, where $K$ is an imaginary quadratic extension of $\Q$. Let $g$ be the theta series associated to $\psi_g$ and let $p$ be a rational prime not dividing its level that splits $K$. Suppose that $g$ satisfies the hypothesis stated in \S \ref{SectionLGr} and recall the definition of $\mathcal L^\mathrm{Gr}(\ad^0(g_\alpha))$ given in Definition \ref{defLGr}. In this section we compute $\mathcal L^\mathrm{Gr}(\ad^0(g_\alpha))$ modulo $\Q^\times$. This allows us to prove our main result in this case.

Denote by $\psi_g'$ the character of $G_K$ defined as follows: if $\sigma \in G_K$, define $\psi_g'(\sigma) = \psi_g(\tau \sigma \tau^{-1})$, where we recall that $\tau$ is a lift to $G_\Q$ of the nontrivial element of $\Gal(K/\Q)$. Let $\psi = \psi_g/\psi_g'$. In this case it is well known that as $G_{\Q}$-modules
\[
\ad^0(g) = \chi_K \oplus \Ind_{G_K}^{G_\Q}(\psi),
\]
where $\chi_K:G_\Q \to \Q^\times$ is the character of order two attached to $K$ and note that $\Ind_{G_K}^{G_\Q}(\psi)$ can be either irreducible or decompose as the sum of two $1$-dimensional characters. Moreover, since $p$ is unramified in $H$ we have that $p$ is coprime to the conductor of $\psi$. 

Since $p$ splits in $K$ we have $G_{\Q_p} \subset G_K$. Therefore, by Hypothesis (B1), the invariant line of $\ad^0(g)$ where $G_{\Q}$ acts by $\chi_K$ is the unique line of $\ad^0(g)$ where $G_{\Q_p}$ acts trivially. Thus, following Remark \ref{subspaceofad}, we will mainly work with the subspace of $\ad^0(g)$ corresponding to $\chi_K$. In fact, we could do all the calculations of this section using only elements of $K$, as it is done in \cite{Gr}, but we will work with elements of $H$ to show the similarities of these calculations with the ones presented in the next sections.

As in \S\ref{SectionLGr}, $\Delta = \Gal(H/\Q)$, $\fp_1 , \dots , \fp_n$ are the primes of $H$ above $p$, where $\fp_1$ is the prime singled out by the embedding $\bar{\Q} \subset \bar{\Q}_p$, $\Delta_{\fp_i}$ is the decomposition subgroup at $\fp_i$ for every $i$ and $H_1$ is the subfield of $H$ fixed by $\Delta_{\fp_1}$. Choose also $\sigma_1, \dots, \sigma_n \in \Gal(H/K)$ such that $\sigma_i(\fp_1) = \fp_i$, $\sigma_1 = 1$,  and 
\[
\Delta = \bigsqcup_{i = 1}^n \sigma_i\Delta_{\fp_1}.
\]

\subsection{Study of the subspace $C \subset H^1(\Q_\Sigma/\Q,V)$}\label{studyofCKQIpsplit}
Recall that $\pi \in H_1$ is a generator of a power of the prime ideal of $H_1$ below $\fp_1$ which is congruent to $1$ modulo $\fp_i$ for all $i = 2 , \dots , n$  (see Definition \ref{defpi}). By Proposition \ref{piepsilona}, there exists a unique $x \in 1 + p\Z_p$ such that $\pi x \in \mathrm{UnivNorm}(H_\infty)$. The study of the subspace $C$  allows us to give an explicit version of Proposition \ref{f(w)=0KQIpinert}, which gives a condition that $w = (x , 1/\pi , \dots , 1/\pi) \in \prod_{i = 1}^n U_{\fp_i,\Q_p}^1$ must satisfy. Later, we will use this to express $x$ in terms of a $p$-unit of $K$ and compute $\mathcal L^{\mathrm{Gr}}(\ad^0(g_\alpha))$.

\begin{lemma}\label{descriptionCKQIpsplit}
	We have that $C = H^1(\Q_\Sigma/\Q, \chi_K) \subset H^1(\Q_\Sigma/\Q,V)$.
\end{lemma}
\begin{proof}
	In the proof of Proposition \ref{dimLHS} it is proven that $\dim_{\Q_p} H^1(\Q_\Sigma/\Q, \chi_K) = 1$. Moreover, since $p$ splits in $K$, $\lambda\left(H^1(\Q_\Sigma/\Q,\chi_K)\right) \subset H^1(\Q_p, \Q_p) \subset H^1(\Q_p, V/F^+V)$. Thus, the result follows from Remark \ref{lambda(C)=intersection}.
\end{proof}

Let $\theta \in \Q[\Delta] \subset \Q_p[\Delta]$ be the idempotent of $\chi_K$. The next corollary is a consequence of Proposition \ref{f(w)=0KQIpinert}.

\begin{corollary}\label{theta(w)=0}
	Let $w = (x, 1/\pi, \dots , 1/\pi) \in \prod_{i = 1}^n U_{\fp_i,\Q_p}^1$ be as in Definition \ref{defwKQIpinert}. Then, 
	\[
	\theta(w) = 0.
	\]
\end{corollary}
\begin{proof}
	Let $f \in H^1(\Q_\Sigma/\Q,V) \simeq \Hom_\Delta(\prod_{i = 1}^n U_{\fp_i,\Q_p}^1,V)$ be a generator of $C$. Proposition \ref{f(w)=0KQIpinert} imples that $f(w) = 0$.
	By Lemma \ref{descriptionCKQIpsplit}, we have $C = H^1(\Q_\Sigma/\Q,\chi_K) \simeq \Hom_{\Delta}(\prod_{i = 1}^n U_{\fp_i,\Q_p}^1, \chi_K)$ so $f(w) = f(\theta(w)) = 0$. Since $\chi_K$ has multiplicity $1$ in $\prod_{i = 1}^n U_{\fp_i,\Q_p}^1$ it has to be $\theta(w) = 0$.
\end{proof}

\subsection{Computation of $u_1$ and $v_1$}\label{SubsectionuvKQIpsplits}
Recall $u$ and $v$ defined in \S\ref{SectionLan}. We see that $u_1 = 0$ and we give a simple expression for $v_1$. Later, this will allow us to compute $\mathcal L^\mathrm{Gr}(\ad^0(g_\alpha))$ in terms of these element and compare it with $\mathcal L^\mathrm{an}(\ad^0(g_\alpha))$.

Recall the definition of $e_1$ given in Definition \ref{defe}. Since $p$ splits in $K$, we can fix $e_1$ to be a generator of the line of $\ad^0(g)$ where $G_\Q$ acts by $\chi_K$.

\begin{proposition}\label{uKQIpsplit}
	For any $u \in \Hom_{G_\Q}(\ad^0(g), \cO_H^\times \otimes L)$ we have $u_1 = 0$.
\end{proposition}
\begin{proof}
	Noting that $\ad^0(g) = \chi_K \oplus \Ind_{G_K}^{G_\Q}(\psi)$, we obtain the following decomposition: $\Hom_{G_\Q}(\ad^0(g),\cO_H^\times \otimes L) = \Hom_{G_\Q}(\chi_K,\cO_H^\times \otimes L) \oplus \Hom_{G_\Q}(\Ind_{G_K}^{G_\Q}(\psi),\cO_H^\times \otimes L)$. Since $\cO_K^\times$ is a finite group, $\Hom_{G_\Q}(\chi_K,\cO_H^\times \otimes L) = 0$ and the result follows. 
\end{proof}

And now we find $v$.

\begin{lemma}\label{thetapineq0KQIpsplit}
	We have that $\theta(\pi) \neq 0$ as an element of $\cO_H[1/p]^\times/p^\Z \otimes L$.
\end{lemma}
\begin{proof}
	Note that since $H_1 = H^{\Delta_{\fp_1}}$ and $\pi$ generates a power of $\fp_1'$, we see that $\pi\cO_H$ also generates a power of $\fp_1$. Therefore, $\sigma_i(\pi)$ generates (a power) of $\fp_i$ for every $i$. Hence, it follows from Dirichlet's S-unit theorem that the set of global units together with $\sigma_i(\pi)$ for all $i$ generates $\cO_H[1/p]^\times/p^\Z \otimes L$. 
	
	For the sake of contradiction assume that $\theta(\pi) = 0$. Then, for every $\sigma_i$ we would have $\theta(\sigma_i\pi) = \sigma_i \theta(\pi) = 0$ which implies $\Hom_{G_\Q}(\chi_K, \cO_H[1/p]^\times/p^\Z \otimes L) = \Hom_{G_\Q}(\chi_K, \cO_H^\times \otimes L) = 0$, because $\cO_K^\times$ is finite. This is a contradiction: since $p$ splits in $K$ it follows from Dirichlet $S$-unit theorem that $\dim_L \Hom_{G_\Q}(\chi_K, \cO_H[1/p]^\times/p^\Z \otimes L) = 1$. Thus $\theta(\pi) \neq 0$ as we wanted to see.
\end{proof}

\begin{proposition}\label{vQIpsplit}
	We can choose $v \in \Hom_{G_\Q}(\ad^0(g), \cO_H[1/p]^\times/p^\Z \otimes L)$ such that $v_1 = \theta(\pi)$. Moreover, $\{ u , v \}$ is a basis of $\Hom_{G_\Q}(\ad^0(g), \cO_H[1/p]^\times/p^\Z \otimes L)$ for any $u$ as above.
\end{proposition}
\begin{proof} 
	Define $v \in \Hom(\ad^0(g),\cO_H[1/p]^\times/p^\Z \otimes L)$ as $v(e_1) = \theta(\pi)$ and $v(\Ind_{G_K}^{G_\Q}(\psi)) = 0$ and note that $v$ is $G_{\Q}$-equivariant. Moreover, for any $u \in \Hom_{G_\Q}(\ad^0(g), \cO_H^\times \otimes L)$ we have that $u$ and $v$ are linearly independent. Indeed, by Proposition \ref{uKQIpsplit}, $u_1 = 0$ while we have $v_1 = \theta(\pi) \neq 0$ in $\cO_H[1/p]^\times/p^\Z \otimes L$ by Lemma \ref{thetapineq0KQIpsplit}.
\end{proof}

\subsection{Computation of $\mathcal L^{\mathrm{Gr}}(\ad^0(g_\alpha))$}
Before explaining how we will proceed note the following.
\begin{remark}\label{padiclog}
	For every $i$, the $p$-adic logarithm gives an isomorphism $\log_p: U_{\fp_i,\Q_p}^1 \to H_{\fp_i}$ wich is $\Delta_{\fp_i}$-equivariant. This induces an isomorphism $\log_p: \prod_{i = 1}^n U_{\fp_i,\Q_p}^1 \to \prod_{i = 1}^n H_{\fp_i}$ which is $\Delta$-equivariant. We will work with the image of this map.
\end{remark}

Let $w$ be as in Definition \ref{defwKQIpinert}. We use Corollary \ref{theta(w)=0} to determine the value of $x \in 1 + p\Z_p$. This allows us to compute $\mathcal L^{\mathrm{Gr}}(\ad^0(g_\alpha))$ in terms of $v_1$.

\begin{corollary}\label{LGr=LanKQIpinert}
	Let $v$ be as in \S\ref{SubsectionuvKQIpsplits}. Then, 
	\[
	\mathcal{L}^{\mathrm{Gr}}(\ad^0(g_\alpha)) \equiv \log_p(v_1) \quad (\mathrm{mod} \, \Q^\times).
	\]
	Thus, if $g$ satisfies Hypothesis (A1-2-3),
	\[
	\mathcal{L}^{\mathrm{Gr}}(\ad^0(g_\alpha)) \equiv \mathcal{L}^{\mathrm{an}}(\ad^0(g_\alpha)) \quad (\mathrm{mod} \, L^\times).
	\]
\end{corollary}
\begin{proof}
	By Corollary \ref{theta(w)=0}, we have $\theta(\log_p(w)) = 0$. Using the natural identification $\prod_{i = 1}^n H_{\fp_i} \simeq \Ind_{\Delta_{\fp_1}}^{\Delta}(H_{\fp_1})$, which is isomorphic to the regular representation of $\Delta$, we can compute explicitly $\theta(\log_p(w))$. The first component of $\theta(\log_p(w)) = 0$ (i.e. the component corresponding to $H_{\fp_1}$) gives
	\[
	\#\Delta_{\fp_1}\log_p(x) = \sum_{i = 2}^n \sum_{\delta \in \Delta_{\fp_1}} \chi_K\left((\sigma_i\delta)^{-1}\right)\log_p(\sigma_i \pi). 
	\]
	Therefore, 
	\[
	\mathcal L^{\mathrm{Gr}} (\ad^0(g_\alpha)) \equiv \log_p(x) + \log_p(\pi) \equiv \log_p(v_1) \quad (\mathrm{mod} \, \Q^\times),
	\]
	where we used Proposition \ref{vQIpsplit} for the last congruence. To prove the equality between $\mathcal L$-invariants (modulo $L^\times$) use Proposition \ref{uKQIpsplit} and Theorem \ref{main1'}.
\end{proof}

\section{The case of imaginary quadratic fields in which $p$ is inert}\label{SectionIQpinert}

Consider the notation given in \S\ref{SectionIntroduction} and \S\ref{SectionLGr}. Consider a $1$-dimensional character $\psi_g: G_K \to \C^\times$, where $K$ is an imaginary quadratic extension of $\Q$. Let $g$ be the theta series associated to $\psi_g$ and let $p$ be a rational prime not dividing its level that is inert in $K$. Suppose that $g$ satisfies the hypothesis stated in \S\ref{SectionLGr} and recall the definition of $\mathcal L^\mathrm{Gr}(\ad^0(g_\alpha))$ given in Definition \ref{defLGr}. In this section we compute $\mathcal L^\mathrm{Gr}(\ad^0(g_\alpha))$ modulo $\Q^\times$. This allows us to prove our main result in this case.

Recall the definition of $\psi = \psi_g/\psi_g'$ given in the beginning of \S\ref{SectionIQpsplits}. It is well known that as $G_{\Q}$-modules
\begin{equation}\label{decadKQIpinert}
\ad^0(g) = \chi_K \oplus \Ind_{G_K}^{G_\Q}(\psi),
\end{equation}
where $\chi_K:G_\Q \to \Q^\times$ is the character of order two attached to $K$. Moreover, since $p$ is unramified in $H$ we have that $p$ is coprime with the conductor of $\psi$.

Since $p$ is inert in $K$ we have $G_{\Q_p} \not\subset G_K$. Therefore, $\Ind_{G_K}^{G_\Q}(\psi)$ contains the line of $\ad^0(g)$ where $G_{\Q_p}$ acts trivially. Thus, most of the calculations will be related with the subspace $\Ind_{G_K}^{G_\Q}(\psi)$ (see Remark \ref{subspaceofad}). To simplify the notation, let $\varrho$ be the direct summand of $\varrho_{\ad^0(g)}$ corresponding to $\Ind_{G_K}^{G_\Q}(\psi)$ and let $\theta$ be its idempotent. We can assume that $\varrho$ is irreducible. Otherwise, following the discussion of the proof of Proposition \ref{dimRHS} and using Hypothesis (B2), $\varrho$ has a direct summand equal to $\chi_{K'}$, where $K'$ is an imaginary quadratic field where $p$ splits and $\chi_{K'}$ is the character of order two attached to it. Then, it is well known that $\ad^0(g)$ can be seen as the adjoint representation of the theta series of a character of quadratic imaginary field where $p$ splits, which is a case that we already studied in \S\ref{SectionIQpsplits}.

As in \S\ref{SectionLGr}, $\Delta = \Gal(H/\Q)$, $\fp_1 , \dots , \fp_n$ are the primes of $H$ above $p$, where $\fp_1$ is the prime singled out by the embedding $\bar{\Q} \subset \bar{\Q}_p$, $\Delta_{\fp_i}$ is the decomposition subgroup at $\fp_i$ for every $i$ and $H_1$ is the subfield of $H$ fixed by $\Delta_{\fp_1}$. Choose also $\sigma_1, \dots, \sigma_n \in \Gal(H/K)$ such that $\sigma_i(\fp_1) = \fp_i$, $\sigma_1 = 1$,  and 
\[
\Delta = \bigsqcup_{i = 1}^n \sigma_i\Delta_{\fp_1}.
\]

\begin{remark}\label{sizedec=2}
	Note that $H$ is the ring class field associated to the ring class character $\psi$ and $(p)$ is a principal ideal of $K$ coprime to the conductor of $\psi$. Therefore, $(p)$ splits completely in $H$. Thus, $n = [H:K]$, $\Delta_{\fp_i}$ is a group of order $2$ for every $i$ and $\Gal(H/K) = \{\sigma_1, \dots \sigma_n\}$. 
\end{remark}

\subsection{Study of the subspace $C \subset H^1(\Q_\Sigma/\Q,V)$} We proceed as in \S\ref{studyofCKQIpsplit}. Note that the decomposition of $V$ given in \eqref{decadKQIpinert} induces the following decomposition
\begin{equation}\label{DecH^1KQIpinert}
H^1(\Q_\Sigma/\Q, V) = H^1(\Q_\Sigma/\Q,\chi_K) \oplus H^1(\Q_\Sigma/\Q, \Ind_{G_K}^{G_\Q}(\psi)).
\end{equation}

\begin{lemma}\label{f_2not0}
	Let $f \in H^1(\Q_\Sigma/\Q,V)$ be a generator of $C$. Following \eqref{DecH^1KQIpinert} write
	\[
	f = f_1 + f_2
	\]
	with $f_1 \in H^1(\Q_\Sigma/\Q, \chi_K)$ and $f_2 \in H^1(\Q_\Sigma/\Q, \Ind_{G_K}^{G_\Q}(\psi))$. Then $f_2 \neq 0$.
\end{lemma}
\begin{proof}
	For the sake of contradiction suppose that $f_2 = 0$. Then, $C = H^1(\Q_\Sigma/\Q,\chi_K)$ (we are using that $H^1(\Q_\Sigma/\Q, \chi_K)$ has dimension $1$ as proven in Proposition \ref{dimLHS}). But since $p$ is inert in $K$, $\chi_K$ is not the line where $G_{\Q_p}$ acts trivially so
	\[
	\lambda(C) = \lambda(H^1(\Q_\Sigma/\Q,\chi_K)) \not\subset H^1(\Q_p, \Q_p),
	\]
	where we are viewing $H^1(\Q_p, \Q_p) \subset H^1(\Q_p, V/F^+V)$. This is a contradiction with Corollary \ref{lambda(C)}.
\end{proof}

The following corollary is a consequence of Proposition \ref{f(w)=0KQIpinert}. Before stating it we need to introduce the following notation. Note that, by Remark \ref{sizedec=2}
\[
\theta = \frac{1}{n}\sum_{i = 1}^n (\psi + \psi^{-1})(\sigma_i)\sigma_i \in L[\Delta] \subset \Q_p[\Delta].
\]
The proof of Proposition \ref{dimLHS} shows that $\varrho$ has multiplicity $1$ in $\cO_H^1 \otimes_\Z L$. Since $\varrho$ contains the line of $\ad^0(g)$ where $G_{\Q_p}$ acts trivially we can fix the following global unit.
\begin{definition}\label{defvarepsilon}
	Fix $\varepsilon \in \cO_{H_1}^1$ such that $\theta(\varepsilon) \neq 0$ in $\cO_H^1 \otimes L$. Note that $\theta(\varepsilon)$ generates the line of $(\cO_H^1 \otimes L)^\varrho$ (here and from now on this notation stands for the $\varrho$-isotypical component) where $G_{\Q_p}$ acts trivially.
\end{definition}
Recall that the map $\cO_{H,\Q_p}^1 \to \prod_{i = 1}^n U_{\fp_i,\Q_p}^1$ is $\Delta$-equivariant. Therefore, the image of $\theta(\varepsilon)$ is $(\theta(\varepsilon), \dots, \theta(\varepsilon))$. Moreover, the proof of \cite[Lemma 2.5]{GV} ensures that $(\theta(\varepsilon), \dots, \theta(\varepsilon))$ generates the unique line of $\left(\overline{\cO_{H,\Q_p}^1}\right)^\varrho \subset \prod_{i = 1}^n U_{\fp_i,\Q_p}^1$ where $G_{\Q_p}$ acts trivially.

\begin{corollary}\label{projidempotent}
	Let $w = (x , 1/\pi ,\dots, 1/\pi) \in \prod_{i = 1}^n U_{\fp_i,\Q_p}^1$ be as in Definition \ref{defwKQIpinert}. Then, there exists $b \in \Q_p$ such that
	\[
	\theta(w) = b(\theta(\varepsilon), \dots , \theta(\varepsilon)).
	\]
\end{corollary}
\begin{proof}
	Let $f \in H^1(\Q_\Sigma/\Q,V) \simeq \Hom_{\Delta}(\prod_{i = 1}^n U_{\fp_i,\Q_p}^1, V)$ be a generator of $C$. Write $f = f_1 + f_2$ according to the decomposition \eqref{DecH^1KQIpinert}, where $f_2 \neq 0$ due to Lemma \ref{f_2not0}. By Proposition \ref{f(w)=0KQIpinert}, we have that $f(w) = 0$, so $f_2(w) = 0$ as well. Note that $f_2 \in \Hom_{\Delta}(\prod_{i = 1}^n U_{\fp_i,\Q_p}^1/\overline{\cO_{H,\Q_p}^1}, \Ind_{G_K}^{G_\Q}(\psi))$ factors through
	\[
	\left(\prod_{i = 1}^n U_{\fp_i,\Q_p}^1/\overline{\cO_{H,\Q_p}^1}\right)^\varrho
	\]
	and that in the proof of Proposition \ref{dimLHS} we saw that $\varrho$ has multiplicity $2$ in $\prod_{i = 1}^n U_{\fp_i,\Q_p}^1$ and it has multiplicity $1$ in $\overline{\cO_{H,\Q_p}^1}$. Since $f_2 \neq 0$ but $f_2(w) = f_2(\theta(w)) = 0$ we conclude that 
	\[
	\theta(w) \in \left(\overline{\cO_{H,\Q_p}^1}\right)^{\varrho}.
	\]
	Now we just have to note that $G_{\Q_p}$ acts trivially on $\theta(w)$, so it belongs to the line generated by $(\theta(\varepsilon), \dots , \theta(\varepsilon))$.
\end{proof}
\begin{remark}
	Recall that $U_{\fp_i,\Q_p}^1 = U_{\fp_i}^1 \otimes_{\Z_p} \Q_p$. Therefore, $\prod_{i = 1}^n U_{\fp_i,\Q_p}^1$ is a $\Q_p$-vector space. Note that in the previous corollary, and from now on, we denote the internal operation of the group of principal units by addition and the operation between a scalar of $\Q_p$ and a principal unit by multiplication. We will also use this notation, for example, in Corollary \ref{projidempotentKQRpinert} and in Corollary \ref{thetaw} and we will use an analogous one when working with the $L$-vector spaces $\cO_H^1 \otimes L$ and $\cO_H[1/p]^\times/p^\Z \otimes L$.
\end{remark}

\subsection{Computation of $u_1,\ u_{\beta/\alpha},\ v_1,\ v_{\beta/\alpha}$}\label{SubsectionKQIpinert}
We give a pair $\{u ,v\}$ as in \S\ref{SectionLan} and we use it to compute $u_1,\ u_{\beta/\alpha},\ v_1,\ v_{\beta/\alpha}$ in terms of $\varepsilon$ and $\pi$.

We will define $u,v \in \Hom_{G_\Q}(\ad^0(g), \cO_H[1/p]^\times/p^\Z \otimes L)$ by giving their images in a concrete basis of $\ad^0(g)$. Recall that $\ad^0(g) = \chi_K \oplus \Ind_{G_K}^{G_\Q}(\psi)$. Let $y$ be a generator of the line of $\ad^0(g)$ where $G_{\Q}$ acts by $\chi_K$. Let $s \in \Delta_{\fp_1}$ be such that $s \not \in \Gal(H/K)$. Then, by the definition of induced representation, if we let $x$ be a generator of the line of $\Ind_{G_K}^{G_\Q}(\psi)$ where $G_K$ acts by $\psi$, we have that  $\{x,sx\}$ is a basis of $\Ind_{G_K}^{G_\Q}(\psi)$. 

\begin{definition}\label{baseofInd}
	Define $\{x,sx,y\}$ as above. They form a basis of $\ad^0(g)$.
\end{definition}
\begin{remark}
	Note that we may have to extend scalars in order to define the elements $x,sx$. 
\end{remark}
We proceed to define $u$ and $v$ as in \S\ref{SectionLan}.
\begin{proposition}\label{uKQIpinert}
	The element $u \in \Hom(\ad^0(g), \cO_H^\times \otimes L)$ defined as
	\[
	u(y) = 0,\ u(x) = u_\psi,\ u(sx) = s(u_\psi),
	\]
	where 
	\[
	u_\psi = \frac{1}{n}\sum_{i = 1}^n \sigma_i(\varepsilon) \otimes \psi(\sigma_i^{-1})
	\]
	generates $\Hom_{G_\Q}(\ad^0(g),\cO_H^\times \otimes L)$.
\end{proposition}
\begin{proof}
	The element $u_\psi$ is defined using the idempotent of $\psi$. Hence, it is plain to see that $u$ is $G_{\Q}$-equivariant. In order to see that it generates $\Hom_{G_\Q}(\ad^0(g), \cO_H^\times \otimes L)$ we need to see that $u$ is nonzero. For that, observe that $u(x+sx) = \theta(\varepsilon)$ which is non-zero by Definition \ref{defvarepsilon}.
\end{proof}


In order to define $v$ we will proceed in a similar way. Note that in the definition of $u$ it was important that $\theta(\varepsilon) \neq 0$. The next lemma is the analogous result that we need for $\pi$. Note that $\cO_H^\times \otimes L$ can be viewed as a subspace of $\cO_H[1/p]^\times/p^\Z \otimes L$.

\begin{lemma}\label{thetapineq0}
	We have that $\theta(\pi) \not \in \cO_H^\times \otimes L$ as an element of $\cO_H[1/p]^\times/p^\Z \otimes L$.
\end{lemma}
\begin{proof}
	The proof is essentially the same as the proof of Lemma \ref{thetapineq0KQIpsplit} and it follows from the fact that 
	\[
	\dim_{L}\Hom_{G_\Q}(\ad^0(g),\cO_H^\times \otimes L) = 1, \quad \dim_{L}\Hom_{G_\Q}(\ad^0(g), \cO_H[1/p]^\times/p^\Z \otimes L) = 2.
	\]
\end{proof}

And now we can give $v$.

\begin{proposition}\label{vQIpinert}
	Let $v \in \Hom(\ad^0(g), \cO_H[1/p]^\times/p^{\Z} \otimes L)$ be defined as
	\[
	v(y) = 1,\ v(x) = v_\psi,\ v(sx) = s(v_\psi),
	\]
	where 
	\[
	v_\psi = \frac{1}{n}\sum_{i = 1}^n \sigma_i(\pi) \otimes \psi(\sigma_i^{-1}).
	\]
	and let $u$ be as in Proposition \ref{uKQIpinert}. Then, $v$ is $G_{\Q_p}$-equivariant and $\{u,v\}$ is a basis of the  $\Hom_{G_\Q}(\ad^0(g),\cO_H[1/p]^\times/p^\Z \otimes L)$. 
\end{proposition}
\begin{proof}
	As in the proof of Proposition \ref{uKQIpinert} we see that $v$ is $G_\Q$-equivariant and that $v(x + sx) = \theta(\pi)$. To conclude, the fact that $u$ and $v$ are linearly independent follows from Lemma \ref{thetapineq0}.
\end{proof}
We are left with computing $u_1, v_1, u_{\beta/\alpha}, v_{\beta/\alpha}$. Recall the definition of $e_1, e_{\beta/\alpha}$ given in Definition \ref{defe}. It is a computation to check that, after a scaling of $x$ and $y$, if needed, we can choose $e_1 = x + sx$ and $e_{\beta/\alpha} = (sx - x) + y$. 

\begin{corollary}\label{u_iv_i}
	Let $u$ and $v$ be as in Proposition \ref{uKQIpinert} and Proposition \ref{vQIpinert}. Let $e_1$ and $e_{\beta/\alpha}$ be as above. Following the notation introduced below we have
	\[
	u_1 = \frac{1}{n}\sum_{i = 1}^n \sigma_i(\varepsilon)\otimes(\psi + \psi^{-1})(\sigma_i) = \theta(\varepsilon), \quad u_{\beta/\alpha} = \frac{1}{n}\sum_{i = 1}^n \sigma_i(\varepsilon)\otimes(\psi - \psi^{-1})(\sigma_i).
	\]
	Similarly,
	\[
	v_1 = \frac{1}{n}\sum_{i = 1}^n \sigma_i(\pi)\otimes(\psi + \psi^{-1})(\sigma_i) = \theta(\pi), \quad v_{\beta/\alpha} = \frac{1}{n}\sum_{i = 1}^n \sigma_i(\pi)\otimes(\psi - \psi^{-1})(\sigma_i).
	\]
\end{corollary}

\subsection{Computation of $\mathcal L^\mathrm{Gr} (\ad^0(g_\alpha))$}
Recall that Remark \ref{padiclog} justifies that the $p$-adic logarithm gives the $\Delta$-equivariant isomorphism $\log_p: \prod_{i=1}^n U_{\fp_i,\Q_p}^1 \to \prod_{i = 1}^n H_{\fp_i}$. From now on we will work with the image of this map. Let $w$ be as in Definition \ref{defwKQIpinert}. We will use Corollary \ref{projidempotent} to determine the value of $x \in 1 + p\Z_p$ in terms of $u_1,\ u_{\beta/\alpha},\ v_1$ and $v_{\beta/\alpha}$ which allows us to compute $\mathcal L^{\mathrm{Gr}}(\ad^0(g_\alpha))$ modulo $\Q^\times$ and compare it with $\mathcal L^{\mathrm{an}}(\ad^0(g_\alpha))$.


\begin{lemma}
	Let $b\in \Q_p$ be as in Corollary \ref{projidempotent}. Then, for every $j = 1, \dots , n$ we have
	\begin{equation}\label{conditionsj}
	(\psi + \psi^{-1})(\sigma_j) \log_p(x) - \sum_{i = 2}^n (\psi + \psi^{-1})(\sigma_j\sigma_i) \log_p(\sigma_i\pi) =bn\log_p(\sigma_j^{-1}\theta\varepsilon).
	\end{equation}
\end{lemma}
\begin{proof}
	By Corollary \ref{projidempotent} we have $\theta(\log_p(w)) = b (\log_p(\theta\varepsilon), \dots, \log_p(\theta\varepsilon))$. Then, the desired expression is obtained computing explicitly $\theta(\log_p(w))$, taking the $j$th component (i.e. the component corresponding to $H_{\fp_j} = \sigma_j H_{\fp_1}$) and applying $\sigma_j^{-1}$ to both sides.
\end{proof}

Now we will use \eqref{conditionsj} for $j = 2, \dots, n$ to determine $b$ as above. This value will depend on $u_{\beta/\alpha}$ and $v_{\beta/\alpha}$. 

\begin{proposition}\label{expressionb}
	Let $b$ be as in Corollary \ref{projidempotent}. Then,
	\[
	b = -\frac{\log_p(v_{\beta/\alpha})}{\log_p(u_{\beta/\alpha})}.
	\]
\end{proposition}
\begin{proof}
	Using Corollary \ref{u_iv_i}, we have 
	\[
	\log_p(u_{\beta/\alpha}) = \frac{1}{n}\sum_{j = 2}^n (\psi - \psi^{-1})(\sigma_j^{-1})\log_p(\sigma_j^{-1}\varepsilon)
	\]
	and a similar expression holds for $\log_p(v_{\beta/\alpha})$ as well.
	Now, for $j = 2, \dots, n$ we take \eqref{conditionsj} and multiply both sides by $(\psi - \psi^{-1})(\sigma_j^{-1})/n$. If we add these equalities for $j = 2, \dots , n$ the right hand side equals to 
	\[
	bn\log_p(u_{\beta/\alpha}),
	\]
	where we used that $\theta(u_{\beta/\alpha}) = u_{\beta/\alpha}$.
	Now we compute the left hand side. Note that
	\[
	\sum_{j = 2}^n (\psi^{-1}- \psi)(\sigma_j)(\psi + \psi^{-1})(\sigma_j)\log_p(x) = \sum_{j = 2}^n \left(\psi^{-2}(\sigma_j) - \psi^2(\sigma_j)\right) \log_p(x) = 0.
	\]
	Hence, the left hand side is
	\[
	-\frac{1}{n}\sum_{j = 2}^n \sum_{i = 2}^n (\psi^{-1} - \psi)(\sigma_j)(\psi + \psi^{-1})(\sigma_j\sigma_i)\log_p(\sigma_i \pi)
	\]
	which equals to
	\begin{multline*}
	-\frac{1}{n}\sum_{i = 2}^n \sum_{j = 2}^n \left(\psi(\sigma_i) - \psi^{-1}(\sigma_i) + \psi(\sigma_j)^{-2}\psi(\sigma_i)^{-1} - \psi(\sigma_j)^2\psi(\sigma_i)\right)\log_p(\sigma_i\pi) = \\ = -(n-1)\log_p(v_{\beta/\alpha}) - \frac{1}{n}\sum_{i = 2}^n \left(\left(\psi(\sigma_i)^{-1} - \psi(\sigma_i)\right)\left(\sum_{j = 2}^n \psi(\sigma_j)^2\right)\log_p(\sigma_i\pi)\right).
	\end{multline*}
	Note that $\sum_{i = 2}^n \psi(\sigma_j)^2 = -1$. Indeed, since $\Ind_{G_K}^{G_\Q}(\psi)$ is irreducible, there exists $j$ such that $\psi^2(\sigma_j) \neq 1$ which implies that $\sum_{i = 1}^n \psi^2(\sigma_j) = 0$. Hence, the expression above equals to $-n\log_p(v_{\beta/\alpha})$ and the result follows.
\end{proof}

And now we can use the previous proposition and \eqref{conditionsj} for $j = 1$ to find an expression for $x$ and hence for $\mathcal L^{\mathrm{Gr}}(\ad^0(g))$ up to $\Q^\times$. This allows us to prove the main result in this case.

\begin{corollary}\label{LGr=LanKQIpinert}
	Let $\{u,v\}$ be as in \S\ref{SubsectionKQIpinert}. We have 
	\[
	\mathcal{L}^{\mathrm{Gr}}(\ad^0(g_\alpha)) \equiv \log_p(v_1) - \frac{\log_p(v_{\beta/\alpha})}{\log_p(u_{\beta/\alpha})}\log_p(u_1) \quad (\mathrm{mod} \, \Q^\times).
	\]
	Thus, if $g$ satisfies Hypothesis (A1-2-3),
	\[
	\mathcal{L}^{\mathrm{Gr}}(\ad^0(g_\alpha)) \equiv \mathcal{L}^{\mathrm{an}}(\ad^0(g_\alpha)) \quad (\mathrm{mod} \, L^\times).
	\]
\end{corollary}
\begin{proof}
	Since $x \in 1 + p\Z_p$ is such that $\pi x \in \mathrm{UnivNorm}(H_\infty)$
	\[
	\mathcal{L}^{\mathrm{Gr}}(\ad^0(g_\alpha)) \equiv  2\log_p(\pi x) = 2\log_p(x) + 2\log_p(\pi) \quad (\mathrm{mod} \, \Q^\times).
	\]
	On the other hand, taking \eqref{conditionsj} with $j = 1$ and the expression for $b$ found in Proposition \ref{expressionb}
	\[
	2\log_p(x) - \sum_{i = 2}^n (\psi + \psi^{-1})(\sigma_i)\log_p(\sigma_i\pi) = -n\frac{\log_p(v_{\beta/\alpha})}{\log_p(u_{\beta/\alpha})}\log_p(\theta\varepsilon).
	\]
	To conclude with the proof of the first statement note that by Corollary \ref{u_iv_i} we have $\log_p(u_1) = \log_p(\theta\varepsilon)$ and $\log_p(v_1) = \log_p(\theta\pi)$. The second statement follows from the first one and from Theorem \ref{main1'}.
\end{proof}

\section{The case of real quadratic fields in which $p$ is inert}\label{SectionRQpinert}

Consider the notation given in \S\ref{SectionIntroduction} and \S\ref{SectionLGr}. Consider a $1$-dimensional character $\psi_g: G_K \to \C^\times$, where $K$ is a real quadratic extension of $\Q$. Let $g$ be the theta series associated to $\psi_g$ and let $p$ be a rational prime not dividing its level that is inert in $K$. Suppose that $g$ satisfies the hypothesis stated in \S\ref{SectionLGr} and recall the definition of $\mathcal L^\mathrm{Gr}(\ad^0(g_\alpha))$ given in Definition \ref{defLGr}. In this section we compute $\mathcal L^\mathrm{Gr}(\ad^0(g_\alpha))$ modulo $\Q^\times$ for $g$. This allows us to prove our main result in this case.

We consider the exact same notation given in the introduction of \S\ref{SectionIQpinert}. By the same argument given there, we can assume that the direct summand of $\varrho_{\ad^0(g)}$ corresponding to $\Ind_{G_K}^{G_\Q}(\psi)$, denoted by $\varrho$, is irreducible. In addition, note that Remark \ref{sizedec=2} also applies for this case.

\subsection{Study of the subspace $C \subset H^1(\Q_\Sigma/\Q,V)$} \label{studyofCKQRpinert}
We proceed as in \S\ref{studyofCKQIpsplit}. Recall the natural isomorphism $H^1(\Q_\Sigma/\Q,V) \simeq \Hom_{\Delta}(\prod_{i = 1}^n U_{\fp_i,\Q_p}^1, \Ind_{G_K}^{G_\Q}(\psi))$ given in the proof of Proposition \ref{dimLHS}. Therefore, in order to determine $C$ it will be enough to determine its image (that we will also denote by $C$) by this isomorphism. 

\begin{definition}\label{deftildevarepsilon}
	Let $\tilde{\varepsilon} \in U_{\fp_1,\Q_p}^1$ be a generator of the line of $U_{\fp_1,\Q_p}^1$ where $\Delta_{\fp_1}$ acts by $\chi_K$ (note that since $p$ is inert in $K$, $\chi_K$ is a $1$-dimensional non-trivial character of $\Delta_{\fp_1}$).
\end{definition}

\begin{lemma}\label{localconditionCKQRpinert}
	Let $f \in \Hom_{\Delta}(\prod_{i = 1}^n U_{\fp_i,\Q_p}^1, \Ind_{G_K}^{G_\Q}(\psi))$ be such that $f(\tilde{\varepsilon},1,\dots, 1) = 0$. Then $f \in C$.
\end{lemma}
\begin{proof}
	Recall the map
	\[
	\mu: H^1(\Q_p, V/F^+V) \to \Hom_{\Delta_{\fp_1}}(H_{\fp_1}^\times, V/F^+V)
	\]
	defined in \eqref{rescftiso}. Following a similar reasoning than the proof Proposition \ref{f(w)=0KQIpinert} we see that $\mu(\lambda(f))(\tilde{\varepsilon}) = 0$ in $V/F^+V$. Since $\tilde{\varepsilon}$ generates the unique line of $H_{\fp_1}^\times \hat{\otimes} \Q_p$ where $G_{\Q_p}$ acts by $\chi_K$, $\mu(\lambda(f)) \in \Hom_{\Delta_{\fp_1}}(H_{\fp_1}^\times, \Q_p)$ so $\lambda(f) \in H^1(\Q_p, \Q_p)$ and the result follows by Remark \ref{lambda(C)=intersection}.
\end{proof}

The previous lemma motivates the following definition. Note that by Remark \ref{sizedec=2} we can write the idempotent of $\varrho$ as
\[
\theta = \frac{1}{n} \sum_{i = 1}^n (\psi + \psi^{-1})(\sigma_i)\sigma_i.
\]
\begin{definition}
	Let $Z_C$ be the $\Delta$-equivariant subspace of $\prod_{i = 1}^n U_{\fp_i,\Q_p}^1$ generated by $\theta(\tilde{\varepsilon}, 1, \dots, 1)$.
\end{definition}

And we use $Z_C$ to describe the $1$-dimensional subspace $C$. As a consequence, we obtain a more explicit version of Proposition \ref{f(w)=0KQIpinert}.

\begin{proposition}\label{descriptionCKQRpinert}
	The subspace $Z_C$ is nonzero and $C \subset \Hom_{\Delta}(\prod_{i = 1}^n U_{\fp_i,\Q_p}^1, \Ind_{G_K}^{G_\Q}(\psi))$ can be identified with
	\[
	\Hom_\Delta\left(\prod_{i = 1}^n U_{\fp_i,\Q_p}^1/ {Z_C} , \Ind_{G_K}^{G_\Q}(\psi)\right).
	\]
	Moreover, we have
	\begin{equation}\label{thetaqKQRpinert}
	\theta(w) \in Z_C.
	\end{equation}
\end{proposition}
\begin{proof}
	We have that $\theta(\tilde{\varepsilon},1, \dots, 1) \neq 0$ since its first component is $2\tilde{\varepsilon}/n \neq 0$. Therefore $Z_C$ is a $\Delta$-equivariant subspace of $\prod_{i = 1}^n U_{\fp_i,\Q_p}^1$ isomorphic to $\Ind_{G_K}^{G_\Q}(\psi)$. On the other hand, the proof of Proposition \ref{dimLHS} shows that $\Ind_{G_K}^{G_\Q}(\psi)$ has multiplicity $2$ in $\prod_{i = 1}^n U_{\fp_i, \Q_p}^1$. Hence, $\Hom_\Delta\left(\prod_{i = 1}^n U_{\fp_i,\Q_p}^1/ {Z_C} , \Ind_{G_K}^{G_\Q}(\psi)\right)$ is naturally a $1$-dimensional subspace of $\Hom_{\Delta}(\prod_{i = 1}^n U_{\fp_i,\Q_p}^1,\Ind_{G_K}^{G_\Q}(\psi))$. Moreover, Lemma \ref{localconditionCKQRpinert} shows that this space is included in $C$. Since $C$ also has dimension $1$, they must be equal. The proof of the second statement is analogous to the proof of Corollary \ref{projidempotent}. 	
\end{proof}

We rewrite \eqref{thetaqKQRpinert} specifying which element of $Z_C$ appears in the expression for $\theta(w)$. Let $\theta_{1,\Delta_{\fp_1}} \in L[\Delta_{\fp_1}] \subset L[\Delta]$ be the idempotent of the trivial representation of $\Delta_{\fp_1}$.
\begin{definition}\label{defz1KQRpinert}
	Fix $\tilde{\sigma} \in \Delta$ such that $\tilde{\sigma}\theta(\tilde{\varepsilon},1, \dots , 1)$ has nontrivial projection on the line of $Z_C$ where $\Delta_{\fp_1}$ acts trivially (recall that $Z_C$ is isomorphic to $\Ind_{G_K}^{G_\Q}(\psi)$). Define 
	\[
	z_1 = \theta_{1,\Delta_{\fp_1}}\tilde{\sigma}\theta(\tilde{\varepsilon},1,\dots, 1).
	\]
	Note that $z_1$ is a generator of the line of $Z_C$ where $G_{\Q_p}$ acts trivially.
\end{definition}
\begin{corollary}\label{projidempotentKQRpinert}
	There exists $c \in \Q_p$ such that
	\[
	\theta(w) = c z_1.
	\]
\end{corollary}
\begin{proof}
	It follows from \eqref{thetaqKQRpinert} and the fact that $G_{\Q_p}$ acts trivially on $w$ and $z_1$.
\end{proof}

\subsection{Computation of $u_1$ and $v_1$}\label{SubsectionKQRpinert}
Recall $u$ and $v$ defined in \S\ref{SectionLan}. Here, we see that $u_1 = 0$ and we give an expression for $v_1$. 

Recall the definition of $e_1$ given in Definition \ref{defe}. Since $p$ is inert in $K$, we can fix $e_1$ to be a generator of the line of $\Ind_{G_K}^{G_\Q}(\psi)$ where $\Delta_{\fp_1}$ acts trivially.

\begin{proposition}\label{uKQRpinert}
	For any $u \in \Hom_{G_\Q}(\ad^0(g), \cO_H^\times \otimes L)$ we have $u_1 = 0$.
\end{proposition}
\begin{proof}
	We have,
	\[
	\Hom_{G_\Q}(\ad^0(g), \cO_H^\times \otimes L) = \Hom_{G_\Q}(\chi_K, \cO_H^\times \otimes L) \oplus \Hom_{G_\Q}(\Ind_{G_K}^{G_\Q}(\psi), \cO_H^\times \otimes L).
	\]
	But, as we saw in the proof Proposition \ref{dimLHS}, $\Hom_{G_\Q}(\Ind_{G_K}^{G_\Q}(\psi), \cO_H^\times \otimes L) = 0$ and therefore $u_1 = u(e_1) = 0$.
\end{proof}

And now we compute $v_1$.
\begin{proposition}\label{vQRpinert}
	We can choose $v \in \Hom_{G_\Q}(\ad^0(g), \cO_H[1/p]^\times/p^\Z \otimes L)$ such that $v_1 = v(e_1) = \theta(\pi)$. Moreover, for any $u$ as above, the elements $\{ u , v \}$ form a basis of $\Hom_{G_\Q}(\ad^0(g), \cO_H[1/p]^\times/p^\Z \otimes L)$.
\end{proposition}
\begin{proof}
	In order to show that there exists $v \in \Hom_{G_\Q}(\ad^0(g), \cO_H[1/p]^\times/p^\Z \otimes L)$ such that $v_1 = \theta(\pi)$ we proceed as in Corollary \ref{u_iv_i}. Moreover, to see that for any $u \in \Hom_{G_\Q}(\ad^0(g), \cO_H \otimes L)$ we have that $u$ and $v$ are linearly independent we proceed as in Lemma \ref{thetapineq0KQIpsplit} and Proposition \ref{vQIpsplit}
\end{proof}

\subsection{Computation of $\mathcal L^{\mathrm{Gr}}(\ad^0(g_\alpha))$} Recall that Remark \ref{padiclog} justifies that the $p$-adic logarithm gives the $\Delta$-equivariant isomorphism $\log_p: \prod_{i=1}^n U_{\fp_i,\Q_p}^1 \to \prod_{i = 1}^n H_{\fp_i}$. From now on we will work with the image of this map. Let $w$ be as in Definition \ref{defwKQIpinert} and $z_1$ as in Definition \ref{defz1KQRpinert}. We will use Corollary \ref{projidempotentKQRpinert} to determine the value of $x \in 1 + p\Z_p$.

The following lemma simplifies a lot the equation appearing in Corollary \ref{projidempotentKQRpinert}.
\begin{lemma}\label{formz1KQRpinert}
	The first component of $\log_p(z_1) \in \prod_{i = 1}^n H_{\fp_i}$ (corresponding to $H_{\fp_1}$) is trivial.
\end{lemma}
\begin{proof}
	A direct computation, using that $\Delta_{\fp_1}$ acts by $\chi_K$ on $\tilde{\varepsilon}$, shows that the $i$th component of $z_1$ is a multiple of $\sigma_i\log_p(\tilde{\varepsilon})$. In particular, the first component is a multiple of $\log_p(\tilde{\varepsilon})$. On the other hand, $\Delta_{\fp_1}$ acts trivially on $z_1$, and therefore on the first component. This shows that the first component has to be $0$ since $\Delta_{\fp_1}$ acts by $\chi_K$ on $\log_p(\tilde{\varepsilon})$ (see Definition \ref{deftildevarepsilon}).
\end{proof}

\begin{corollary}\label{LGr=LanKQRpinert}
	Let $u_1$ and $v_1$ be as in \S\ref{SubsectionKQRpinert}. We have 
	\[
	\mathcal{L}^{\mathrm{Gr}}(\ad^0(g_\alpha)) \equiv \log_p(v_1) \quad (\mathrm{mod} \, \Q^\times).
	\]
	Thus, if $g$ satisfies Hypothesis (A1-2-3),
	\[
	\mathcal{L}^{\mathrm{Gr}}(\ad^0(g_\alpha)) \equiv \mathcal{L}^{\mathrm{an}}(\ad^0(g_\alpha)) \quad (\mathrm{mod} \, L^\times).
	\]
\end{corollary}
\begin{proof}
	By Corollary \ref{projidempotentKQRpinert} we have $\theta(\log_p(w)) = c \log_p(z_1)$ for some $c \in \Q_p$. Computing $\theta(\log_p(w))$ explicitly and taking the first component gives
	\[
	2\log_p(x) - \sum_{i = 2}^n \left(\psi + \psi^{-1}\right)(\sigma_i) \log_p(\sigma_i(\pi)) = 0,
	\]
	where we used Lemma \ref{formz1KQRpinert}. Therefore we can find the value of $2\log_p(x)$ in terms of the $p$-unit $\pi$ and compute the algebraic $\mathcal L$-invariant
	\[
	\mathcal L^{\mathrm{Gr}}(\ad^0(g_\alpha)) \equiv \log_p((\pi x)^2) = 2\log_p(x) + 2\log_p(\pi) = \log_p(\theta(\pi)) \quad (\mathrm{mod} \, \Q^\times).
	\]
	Hence, the first statement follows from Proposition \ref{vQRpinert}. The second statement follows from the first one, Proposition \ref{uKQRpinert} and Theorem \ref{main1'}. 
\end{proof}

\section{Exotic forms}\label{SectionExotic}
Consider the notation given in \S\ref{SectionIntroduction} and \S\ref{SectionLGr}. Let $g$ be an exotic form, i.e. $\ad^0(g)$ is irreducible, and let $p$ be a rational prime not dividing its level. Recall the hypothesis on $g$ stated in \S\ref{SectionLGr}:

\begin{enumerate}
	
	\item[(B1)] $g$ is $p$-regular, i.e. $\alpha/\beta \neq 1$, and
	
	\item[(B2)] $\varrho_g$ is not induced from a character of a real quadratic field in which $p$ splits
\end{enumerate}
(note that only Hypothesis (B1) is relevant in this case). Therefore, we can define  $\mathcal L^{\mathrm{Gr}}(\ad^0(g_\alpha))$ as in  Definition \ref{defLGr}. In addition, suppose that,

\begin{enumerate}
\item[(B3)] $\alpha/\beta \neq -1$.
\end{enumerate}
Under these hypothesis on $g$ we compute $\mathcal L^\mathrm{Gr}(\ad^0(g_\alpha))$ modulo $\Q^\times$. This allows us to prove our main result in this case.

As in \S\ref{SectionLGr}, $\fp_1, \dots , \fp_n$ are the primes of $H$ above $p$, where $\fp_1$ is the prime singled out by the embedding $\bar{\Q} \subset \bar{\Q}_p$. Denote by $\Delta = \Gal(H/\Q)$ and by $\Delta_{\fp_i} \subset \Delta$ the decomposition subgroup at $\fp_i$ for every $i$. Let $\sigma_1 = 1, \dots, \sigma_n \in \Delta$ such that $\sigma_i(\fp_1) = \fp_i$ and
\[
\Delta = \bigsqcup_{i = 1}^n \sigma_i\Delta_{\fp_1}.
\] 
To simplify the notation we will write $\varrho = \varrho_{\ad^0(g)}$, $\chi$ for its character and $\theta_\chi \in L[\Delta] \subset \Q_p[\Delta]$ for its idempotent of $\varrho$, i.e.
\[
\theta_{\chi} = \frac{3}{\#\Delta} \sum_{\sigma \in \Gal(H/\Q)} \chi(\sigma^{-1}) \sigma.
\]

\subsection{Study of the subspace $C \subset H^1(\Q_\Sigma/\Q,V)$} We proceed as
\S\ref{studyofCKQRpinert}: we will identify the image of $C$ by the isomorphism $H^1(\Q_\Sigma/\Q,V) \simeq \Hom_{\Delta}(\prod_{i = 1}^n U_{\fp_i,\Q_p}^1 /\overline{\cO_{H,\Q_p}^1} , V)$. However, the description of $C$ is a little more involved in this case.

\begin{definition}
	Let $\varepsilon_{\alpha/\beta} \in (\cO_H^1 \otimes L)^\varrho$ be a generator of the line of $(\cO_H^1 \otimes L)^\varrho$ where $\Frob_p \in G_{\Q_p}$ acts by $\alpha/\beta$.
\end{definition}

\begin{lemma}\label{localconditionC}
	Let $f \in \Hom_\Delta(\prod_{i = 1}^n U_{\fp_i,\Q_p}^1/\overline{\cO_{H,\Q_p}^1},V)$. If $f(\varepsilon_{\alpha/\beta}, 1, \dots , 1) = 0$ then $f \in C$.
\end{lemma}
\begin{proof}
	Analogous to the proof of Lemma \ref{localconditionCKQRpinert}.
\end{proof}

\begin{definition}
	Let $Z_C$ be the $\Delta$-equivariant subspace of $\prod_{i = 1}^n U_{\fp_i,\Q_p}^1$  generated by $\theta_\chi(\varepsilon_{\alpha/\beta}, 1, \dots, 1)$.
\end{definition}
\begin{proposition}\label{descriptionC}
	We have $Z_C \not \subset \left(\overline{\cO_{H,\Q_p}^1}\right)^{\varrho}$and $C \subset \Hom_{\Delta}(\prod_{i = 1}^n U_{\fp_i}^1/\overline{\cO_H^1}, V)$ can be identified with
	\[
	\Hom_\Delta\left(\prod_{i = 1}^n U_{\fp_i,\Q_p}^1/\langle \overline{\cO_{H,\Q_p}^1}, Z_C \rangle , V\right),
	\]
	where here (and from now on) $\langle \overline{\cO_{H,\Q_p}^1}, Z_C \rangle$ is the subspace of $\prod_{i = 1}^n U_{\fp_i,\Q_p}^1$ generated by $\overline{\cO_{H,\Q_p}^1}$ and $Z_C$.  Moreover,
	\begin{equation}\label{thetawexotic}
	\theta_{\chi}(w) \in \langle \overline{\cO_{H,\Q_p}^1}, Z_C \rangle.
	\end{equation}
\end{proposition}
\begin{proof}
	For the sake of contradiction suppose that $\theta_\chi(\varepsilon_{\alpha/\beta},1,\dots,1) \in \left(\overline{\cO_{H,\Q_p}^1}\right)^{\varrho}$. Then, for all $f \in \Hom_{\Delta}(\prod_{i = 1}^n U_{\fp_i,\Q_p}^1/\overline{\cO_{H,\Q_p}^1}, V)$ we would have that $f(\varepsilon_{\alpha/\beta},1, \dots, 1) = 0$ which implies that $f \in C$ by Lemma \ref{localconditionC}. Hence, $H^1(\Q_\Sigma/\Q,V) = C$ which is a contradiction since $C$ has dimension $1$. Thus, $Z_C$ is a $\Delta$-representation space isomorphic to $\varrho$ such that $Z_C \neq \left(\overline{\cO_{H,\Q_p}^1}\right)^{\varrho}$. On the other hand, the proof of Proposition \ref{dimLHS} shows that $\varrho$ has multiplicity $3$ in $\prod_{i = 1}^n U_{\fp_i,\Q_p}^1$ and multiplicity $1$ in $\overline{\cO_{H,\Q_p}^1}$. This shows that $\Hom_\Delta\left(\prod_{i = 1}^n U_{\fp_i,\Q_p}^1/\langle \overline{\cO_{H,\Q_p}^1}, Z_C \rangle , V\right)$ is $1$-dimensional. In order to see that this subspace is equal to $C$ we use Lemma \ref{localconditionC} and we proceed as in Proposition \ref{descriptionCKQRpinert}. The proof of the second statement is analogous to the proof of Corollary \ref{projidempotent}. 
\end{proof}

Since $w$ is fixed by $\Delta_{\fp_1}$, we can rewrite \eqref{thetawexotic} specifying which elements of $\overline{\cO_{H,\Q_p}^1}$ and of $Z_C$ appear in the expression for $\theta_\chi(w)$. Denote by $\theta_{1,\Delta_{\fp_1}} \in L[\Delta_{\fp_1}] \subset L[\Delta]$ the idempotent of the trivial representation of $\Delta_{\fp_1}$. 

\begin{definition}\label{defz1}
	Fix $\tilde{\sigma} \in \Gal(H/\Q)$ such that
	$\tilde{\sigma} \theta_{\chi} (\varepsilon_{\alpha/\beta}, 1, \dots , 1)$ has nontrivial projection on the line of $Z_C$ where $\Delta_{\fp_1}$ acts trivially. Define
	\[
	z_1 = \theta_{1, \Delta_{\fp_1}} \tilde{\sigma} \theta_\chi(\varepsilon_{\alpha/\beta}, 1, \dots , 1).
	\]
	Note that $z_1$ is a generator of the line of $Z_C$ where $\Delta_{\fp_1}$ acts trivially.
\end{definition}

\begin{definition}
	Fix $\varepsilon_1 \in \cO_{H_1}^1$ such that $\theta_\chi(\varepsilon_1) \neq 0$ in $\cO_H^\times \otimes L$.
\end{definition}

\begin{corollary}\label{thetaw}
	There exist $b,c \in \Q_p$ such that
	\begin{equation}\label{thetaw=ez}
	\theta_\chi(w)= b\left(\theta_\chi(\varepsilon_1), \dots , \theta_\chi(\varepsilon_1)\right) + c z_1. 
	\end{equation}
\end{corollary}
\begin{proof}
	Analogous to the proof of Corollary \ref{projidempotent}.
\end{proof}

\subsection{Computation of $u_1,\ u_{\beta/\alpha},\ v_1,\ v_{\beta/\alpha}$}\label{Subsectionuvexotic}
Here, we give a pair $\{u ,v\}$ as in \S\ref{SectionLan} and we use it to compute $u_1,\ u_{\beta/\alpha},\ v_1,\ v_{\beta/\alpha}$ in terms of $\varepsilon_1$ and $\pi$. 

Let $\theta_{1, \Delta_{\fp_1}}, \theta_{\alpha/\beta, \Delta_{\fp_1}}, \theta_{\beta/\alpha, \Delta_{\fp_1}} \in L[\Delta_{\fp_1}] \subset L[\Delta]$ be the idempotents corresponding to the 1-dimensional representations of $\Delta_{\fp_1}$ such that $\Fr_p \in G_{\Q_p}$ acts by $1, \alpha/\beta$ and $\beta/\alpha$ respectively. Following Definition \ref{definitionuv} we have to choose generators $e_1$ and $e_{\beta/\alpha}$ of the lines $L$ and $L^{\beta/\alpha}$ of $\ad^0(g)$. We do it in the following way. Let $e_1$ be any generator of $L$. Since $\ad^0(g)$ is irreducible, there exists $\sigma \in \Delta$ such that $\theta_{\beta/\alpha, \Delta_{\fp_1}}\sigma(e_1) \neq 0$. Without loss of generality assume that $\sigma = \sigma_2^{-1}\delta$, with $\delta \in \Delta_{\fp_1}$ (this choice will simplify some calculations later). Define $e_{\beta/\alpha} = \theta_{\beta/\alpha, \Delta_{\fp_1}}\sigma_2^{-1}(e_1)$. 

Now we give $u$. 

\begin{proposition}\label{tentativeu}
	There exists $u$, a generator of $\Hom_{G_\Q}(\ad^0(g), \cO_H^\times \otimes L)$, such that $u(e_1) = u_1 = \theta_\chi(\varepsilon_1)$ and $u(e_{\beta/\alpha}) = u_{\beta/\alpha} = \theta_{\beta/\alpha,\Delta_{\fp_1}}\sigma_2^{-1}\theta_\chi(\varepsilon_1)$.
\end{proposition}
\begin{proof}
	By the definition of $\varepsilon_1$, there exists such a $u$ satisfying that $u(e_1) = \theta_\chi(\varepsilon_1) \neq 0$. The second condition follows from the $\Delta$-equivariance of $u$.
\end{proof}

And now we can determine $v$. Let $u$ be as in Proposition \ref{tentativeu}
\begin{proposition}\label{tentativev}
	There exists $v \in \Hom_{G_\Q}(\ad^0(g), \cO_H[1/p]^\times/p^\Z \otimes L)$ such that $v(e_1) = v_1 = \theta_\chi(\pi)$ and $v(e_{\beta/\alpha}) = v_{\beta/\alpha} = \theta_{\beta/\alpha,\Delta_{\fp_1}}\sigma_2^{-1}\theta_\chi (\pi)$. Moreover, $v$ is such that  $\{ u , v \}$ is a basis of $\Hom_{G_\Q}(\ad^0(g), \cO_H[1/p]^\times/p^\Z \otimes L)$.
\end{proposition}
\begin{proof}
	The proof of the first statement is analogous to the proof of Proposition \ref{tentativeu}. Proceeding as in Lemma \ref{thetapineq0} we see that $u$ and $v$ are linearly independent which shows the second statement.
\end{proof}

\subsection{Computation of $\mathcal L^\mathrm{Gr}(\ad^0(g_\alpha))$} Recall that Remark \ref{padiclog} justifies that the $p$-adic logarithm gives the $\Delta$-equivariant isomorphism $\log_p:\prod_{i=1}^n U_{\fp_i,\Q_p}^1 \to \prod_{i = 1}^n H_{\fp_i}$. From now on we will work with the image of this map. Let $w$ be as in Definition \ref{defwKQIpinert} and $z_1$ as in Definition \ref{defz1}. We will use Corollary \ref{thetaw} to determine the value of $x \in 1+ p\Z_p$ in terms of $u_1,u_{\beta/\alpha}, v_1, v_{\beta/\alpha}$.


\begin{lemma}\label{formz1}
	Let $z_1$ be as in Definition \ref{defz1}. The $i$th component of $\log_p(z_1)$ (i.e. the component corresponding to $H_{\fp_i}$) is a multiple of $\sigma_i \log_p(\varepsilon_{\alpha/\beta})$. In particular, the first component of $\log_p(z_1)$ is trivial.
\end{lemma}
\begin{proof}
	Analogous to the proof of Lemma \ref{formz1KQRpinert}.
\end{proof}

\begin{corollary}\label{1component}
	Let $b \in \Q_p$ be as in Corollary \ref{thetaw}. We have 
	\begin{equation}\label{firstcomponentexotic}
	\frac{3}{\#\Delta}\left(\#\Delta_{\fp_1} \log_p(x) - \sum_{i = 2}^n \sum_{\delta \in \Delta_{\fp_1}} \chi((\sigma_i \delta)^{-1}) \log_p(\sigma_i\pi)\right) = b \log_p(\theta_\chi\varepsilon_1).
	\end{equation}
\end{corollary}
\begin{proof}
	We have $\theta_\chi(\log_p(w)) = b(\log_p(\theta_\chi\varepsilon_1), \dots, \log_p(\theta_\chi\varepsilon_1)) + c \log_p(z_1)$ by Corollary \ref{thetaw}. Computing the first component of  $\theta_\chi(\log_p(w))$ gives the desired equation once we note that the first component of $c\log_p(z_1)$ is trivial by Lemma \ref{formz1} .
\end{proof}

The corollary above justifies that in order to determine $x$ we just need to find the value of $b$ appearing in Corollary \ref{thetaw} (i.e., we do not need to compute the value of $c$). 

\begin{proposition}\label{b=v/u}
	Let $b \in \Q_p$ be as in Corollary \ref{thetaw}. Then
	\[
	b = -\frac{\log_p(v_{\beta/\alpha})}{\log_p( u_{\beta/\alpha})}.
	\]
\end{proposition}
\begin{proof}
	Proceeding in a similar way than the proof of the previous lemma, the logarithm of the second component of \eqref{thetaw=ez} of Corollary \ref{thetaw} is
	\begin{multline}\label{2component}
	\frac{3}{\#\Delta}\sigma_2\left(\sum_{\delta \in \Delta_{\fp_1}} \chi((\sigma_2 \delta )^{-1}) \log_p(x) - \sum_{i = 2}^n\sum_{\delta \in \Delta_{\fp_1}} \chi((\sigma_2 \sigma_i \delta )^{-1}) \log_p (\sigma_i \pi) \right) = \\ = \sigma_2 \left( b\log_p(\sigma_2^{-1}\theta_\chi\varepsilon_1) + c K \log_p(\varepsilon_{\alpha/\beta}) \right), 
	\end{multline}
where $K \in \Q_p$ (we are using Lemma \ref{formz1} to see that the second component of $c \log_p(z_1)$ is of the form $\sigma_2\left(cK\log_p(\varepsilon_{\alpha/\beta})\right)$). If we apply $\sigma_2^{-1}$ to both sides we obtain elements of $H_{\fp_1} = \log_p(U_{\fp_1,\Q_p}^1)$, which is a $\Delta_{\fp_1}$-representation space. Apply the idempotent $\theta_{\beta/\alpha , \Delta_{\fp_1}}$ to both sides. 
	
	Since $\beta/ \alpha \neq \alpha/\beta$ (by Hypothesis (B3)), Proposition \ref{tentativeu} shows that the right hand side equals to $b \log_p(u_{\beta/\alpha})$. We will now see that the right hand side equals to $-\log_p(v_{\beta/\alpha})$. First note that,
	\[
	\theta_{\beta/\alpha, \Delta_{\fp_1}} (\log_p(x)) = 0.
	\]
	On the other hand, by Proposition \ref{tentativev},
	\[
	v_{\beta/\alpha} = \theta_{\beta/\alpha, \Delta_{\fp_1}} \theta_{\chi}\sigma_2^{-1} \pi
	\]
	By computing
	\[
	-\theta_\chi \sigma_2^{-1}\pi = -\frac{3}{\#\Delta}\sum_{\delta \in \Delta_{\fp_1}} \chi((\delta \sigma_2)^{-1}) \pi - \frac{3}{\#\Delta}\sum_{i = 2}^n\sum_{\delta \in \Delta_{\fp_1}} \chi((\sigma_i \delta \sigma_2)^{-1}) \sigma_i \pi
	\]
	we note that, since $\chi((\sigma_2\sigma_i\delta)^{-1}) = \chi((\sigma_i \delta \sigma_2)^{-1})$ for any $\delta \in \Delta$, the second summand of the right hand side of \eqref{2component} equals to the second summand of the expression above. Moreover, the first summand of the expression above is annihilated by $\theta_{\beta/\alpha,\Delta_{\fp_1}}$. From here it is plain to obtain the desired result.
\end{proof}

Since we determined $b \in \Q_p$ as above, we use \eqref{firstcomponentexotic} to determine $x$ and compute $\mathcal L^\mathrm{Gr}(\ad^0(g_\alpha))$ modulo $\Q^\times$ proving the main theorem in this last case.

\begin{corollary}
	Let $\{u, v\}$ be as in \S\ref{Subsectionuvexotic}. We have 
	\[
	\mathcal{L}^{\mathrm{Gr}}(\ad^0(g_\alpha)) \equiv \log_p(v_1) - \frac{\log_p(v_{\beta/\alpha})}{\log_p(u_{\beta/\alpha})}\log_p(u_1) \quad (\mathrm{mod} \, \Q^\times).
	\]
	Thus, if $g$ satisfies Hypothesis (A1-2-3),
	\[
	\mathcal{L}^{\mathrm{Gr}}(\ad^0(g_\alpha)) \equiv \mathcal{L}^{\mathrm{an}}(\ad^0(g_\alpha)) \quad (\mathrm{mod} \, L^\times).
	\]
\end{corollary}
\begin{proof}
	Note that by Proposition \ref{tentativeu} we have $\log_p(u_1) = \log_p(\theta_\chi\varepsilon_1)$ and by Proposition \ref{tentativev}
	\[
	\log_p (v_1) - \#\Delta_{\fp_1}\log_p(\pi) = \sum_{i = 2}^n \sum_{\delta \in \Delta_{\fp_1}} \chi((\sigma_i \delta)^{-1}) \log_p(\sigma_i(\pi)).
	\]
	Substituting this to the equation of Corollary \ref{1component} and using Proposition \ref{b=v/u} for the expression for $b$ gives
	\[
	\frac{3}{\#\Delta}(\#\Delta_{\fp_1})(\log_p(x) + \log_p (\pi)) = \log_p(v_1) - \frac{\log_p(v_{\beta/\alpha})}{\log_p(u_{\beta/\alpha})}\log_p(u_1).
	\]
	The first statement follows from this equation while for the second one we also have to use Theorem \ref{main1'}.
\end{proof}

\end{document}